\documentclass[11pt]{amsart}
\usepackage{amscd,amsfonts,amssymb,amsmath}
\usepackage[margin=3.4cm]{geometry}
\usepackage[all,cmtip]{xy}

\usepackage{amssymb}

\newtheorem{theorem}{Theorem}[section]
\newtheorem{cor}[theorem]{Corollary}
\newtheorem{lemma}[theorem]{Lemma}

\theoremstyle{definition}
\newtheorem{definition}[theorem]{Definition}

\theoremstyle{plain}
\numberwithin{equation}{section}

\def \Z{\mathbb Z}

\newcommand{\secref}[1]{Section~\ref{#1}}
\newcommand{\thmref}[1]{Theorem~\ref{#1}}
\newcommand{\lemref}[1]{Lemma~\ref{#1}}

\newcommand{\corref}[1]{Corollary~\ref{#1}}

\newcommand{\eqnref}[1]{~{\textrm(\ref{#1})}}
\numberwithin{equation}{subsection}

\begin{document}
\title[Commutator Subgroups of Singular Braid Groups]{Commutator Subgroups of Singular  Braid Groups}
\author[Soumya Dey]{Soumya Dey}
\author[Krishnendu Gongopadhyay]{Krishnendu Gongopadhyay}
 \address{Indian Institute of Science Education and Research (IISER) Bhopal, 
Bhopal Bypass Road, Bhauri
Bhopal 462 066
Madhya Pradesh, India} 
\email{soumya.sxccal@gmail.com} 
\address{Indian Institute of Science Education and Research (IISER) Mohali, Sector 81,  SAS Nagar, P. O. 
Manauli, Punjab 140306, India}
\email{krishnendu@iisermohali.ac.in, krishnendug@gmail.com} 
\date{\today}

\begin{abstract}
The singular braids with $n$ strands, $n \geq 3$,  were introduced independently by Baez and Birman. It is known that the monoid formed by the singular braids is embedded in a group that is known as singular braid group, denoted by $SG_n$. There has been another generalization of braid groups, denoted by $GVB_n$, $n \geq 3$, which was introduced by Fang as a group of symmetries behind quantum quasi-shuffle structures. The group $GVB_n$ simultaneously generalizes the classical braid group, as well as the virtual braid group on $n$ strands. 

We investigate the commutator subgroups $SG_n'$ and $GVB_n'$ of these generalized braid groups.  We prove that $SG_n'$ is finitely generated if and only if $n \ge 5$, and $GVB_n'$ is finitely generated if and only if $n \ge 4$. Further, we show that both $SG_n'$ and $GVB_n'$ are perfect if and only if $n \ge 5$.
\end{abstract}
\maketitle

\section{Introduction}

 The commutator subgroup or derived subgroup of a group $G$ is the subgroup $G'$ generated by the elements of the form $x^{-1} y^{-1} x y$. A quotient group $G/N$ is abelian if and only if $G' \leq N$. Thus the group $G'$ distinguishes the abelian factor groups of $G$ from non-abelian ones. The quotient $G/G'$  is isomorphic to $H_1(G,\Z)$, the first homology group of $G$ with integer coefficients. So, given a group $G$, the structure of its commutator subgroup $G'$ reveals crucial information. 

\medskip Throughout the following, we shall assume that $n \ge 3$, unless stated otherwise. The braid group on $n$ strands $B_n$, classically known as Artin's braid group, is a central object of investigation due to their importance in several branches of mathematics,  e.g. refer surveys \cite{paris}, \cite{bibr}. Recall that \emph{Artin's braid group} $B_n$ is generated by a set of $(n-1)$ generators: $\{ ~ \sigma_i ~ | ~ i=1, 2, \ldots, n-1 ~ \}$, satisfying the following set of defining relations that we call as \emph{braid relations}: 
	$$\sigma_i \sigma_j=\sigma_j \sigma_i, \hbox{ if } |i-j|>1;$$
	$$\sigma_i \sigma_{i+1} \sigma_i=\sigma_{i+1} \sigma_i \sigma_{i+1}.$$ 
The commutator subgroup $B_n'$ of $B_n$ is well-known from the work of Gorin and Lin \cite{gl} who obtained a finite presentation for  $B_n'$. Simpler presentation for $B_n'$ was obtained by Savushkina \cite{sa}. Several authors have investigated commutators of  larger classes of spherical Artin groups, e.g.  \cite{zinde}, \cite{mr},  \cite{orevkov}.

\medskip There have been several generalizations of Artin's braid groups in the literature. Many of these generalized braid groups are of importance in their own rights, e.g.  \cite{bel1}, \cite{bel2}, \cite{dam}, \cite{vve2}, \cite{vve1}. In this paper, we consider two such families of generalized braid groups and investigate their commutator subgroups.  We briefly introduce them below and note our main results. The relationship between several families of braid groups has been discussed in \secref{prel} and sketched in the diagram \eqnref{gvbndiagram}.

\medskip  The geometric notion of singular braids was introduced independently by Baez in \cite{bae} and Birman in \cite{bir}. It was shown that such braids form a monoid. It is shown in \cite{fkr} that the Baez-Birman monoid on $n$ strands is embedded in a group $SG_n$ that is now known as \emph{singular braid group} on $n$ strands. The group $SG_n$ has a presentation that has the same set of generators and defining relations as the monoid-presentation for the Baez-Birman monoid, with the only additional property that the elements are invertible in $SG_n$. The \emph{singular braid group} $SG_n$ is generated by a set of $2(n-1)$ generators: $\{ ~ \sigma_i,~ \rho_i ~ | ~ i=1, 2, \ldots, n-1 ~ \}$, where $\sigma_i$  satisfy the braid relations as above, $\rho_i$ satisfy the commuting relations: 
\begin{equation}  \label{mm5} \rho_i \rho_j = \rho_j \rho_i, \hbox{ if } |i-j|>1; \end{equation}
and in addition there are the following mixed relations among $\sigma_i$, $\rho_i$: 
		 \begin{equation}  \label{mm2} \rho_i \sigma_{i+1} \sigma_i=\sigma_{i+1} \sigma_i \rho_{i+1};\end{equation}
	 \begin{equation}  \label{mm3} \rho_{i+1} \sigma_i \sigma_{i+1}=\sigma_i \sigma_{i+1} \rho_i;\end{equation}
 \begin{equation} \label{mm1} \sigma_i \rho_j = \rho_j \sigma_i, \hbox{ if } i=j \hbox{ or } |i-j|>1. \end{equation}

Singular braids are related to finite type invariants of knots and links, e.g. \cite[Chapter 3]{bel} and it is a natural problem to investigate their algebraic and geometric properties to understand these invariants. Investigation of topological properties of singular braids was initiated in \cite{vve0}.  We refer to \cite{vve2, vve1, bel} for more on singular braid groups and also other generalized braid groups. We prove the following. 

\begin{theorem}\label{sgnmainth}
	 Let $SG_n'$ denote the commutator subgroup of the singular braid group $SG_n$.
	
	\begin{itemize}
		
		\item[(i)] $SG_n'$ is finitely generated for all  $~ n \ge 5$. Further, for $~ n \ge 5$, the rank of $SG_n'$ is at most $2n - 4$.
		
		\medskip \item[(ii)] $SG_3'$ and $SG_4'$ are not finitely generated.
		
		\medskip \item[(iii)] $SG_n'$ is perfect if and only if $n \ge 5$.
	\end{itemize}
\end{theorem}

\medskip  We consider another family of generalized braid groups, that was introduced by Fang in \cite{fang}. He constructed this generalization as a group of symmetries behind quantum quasi-shuffle structures. Fang's \emph{generalized virtual braid group} $GVB_n$ is generated by a set of $2(n-1)$ generators: $\{ ~ \sigma_i,~ \rho_i ~ | ~ i=1, 2, \ldots, n-1 ~ \}$, where $\sigma_i$ satisfy braid relations, $\rho_i$ also satisfy braid relations, and in addition there are mixed relations \eqnref{mm2},  \eqnref{mm3} as above, and 
 \begin{equation} \label{mm4} \sigma_i \rho_j = \rho_j \sigma_i, \hbox{ if } |i-j|>1.  \end{equation}
The group $GVB_n$ simultaneously generalizes both Artin's braid group $B_n$ and the virtual braid group $VB_n$. We prove the following. 

\begin{theorem}\label{gvbnmainth} Let $GVB_n'$ denote the commutator subgroup of the generalized virtual braid group $GVB_n$.

	\begin{itemize}
		
		\item[(i)] $GVB_n'$ is finitely generated for all  $~ n \ge 4$. Further, for $n \ge 5$, the rank of $GVB_n'$ is at most $3n - 7$. 
		
		\medskip \item[(ii)] $GVB_3'$ is not finitely generated.

		\medskip \item[(iii)] $GVB_n'$ is perfect if and only if $n \geq 5$. 
	\end{itemize}
\end{theorem}

Recall that the virtual braid group $VB_n$ on $n$ strands is a certain generalization of the classical braid group $B_n$ . It was introduced by L.~Kauffman in \cite{lk1},  also see \cite{kala}, \cite{ba}. Virtual braids are  subject of curiosity due to their connection with knot invariants, see \cite{fkv}. Recently commutator subgroups of the virtual braid groups $VB_n$ have been investigated in \cite{bgn1}.  
\thmref{gvbnmainth}  shows that algebraically $GVB_n'$ has similar properties as compared to $VB_n'$, as obtained in \cite{bgn1}. On the other hand,  \thmref{sgnmainth}  shows a contrast between $SG_n'$ and the groups $VB_n'$ and $GVB_n'$, if we compare finite generation for $n=4$.

\medskip The proofs of the theorems are obtained by application of the Reidemeister-Schreier algorithm. 
 Similar tools have been used in \cite{bgn1} and \cite{dg1} to obtain finite generations of the virtual and the welded braid groups.  However, there are certain technical differences between the approaches of this paper and the above ones.  We explain it below. 

\medskip Note that the initial presentation for $VB_n'$ in \cite{bgn1}, or $WB_n'$  in \cite{dg1}, had generating set indexed by $\Z \times \Z_2$. Due to the presence of the simple order two group $\Z_2$ in the index set, the elimination process in \cite{bgn1} or \cite{dg1} was an one-variable process.  On the other hand, the initial generating sets  for both $GVB_n'$ and $SG_n'$ obtained by the Reidemeister-Schreier algorithm are indexed by $\Z \times \Z$. Consequently, the elimination process of the generators is a two-variable process here and seemingly more complicated. However, as we shall see,  this complexity can be handled successfully and we can avoid this technical difficulty to prove the main theorems.

\medskip \thmref{sgnmainth}  and  \thmref{gvbnmainth} may be viewed as the first steps towards understanding of the finiteness properties of $SG_n$ and $GVB_n$ respectively. Using the method used in this paper, we have been unable to produce information about other finiteness properties of $GVB_n$ and $SG_n$; it would be a curious problem to study the same. We note here that the finiteness properties of $VB_n$ is also unknown, except for an analogous result to \thmref{gvbnmainth} in \cite{bgn1}. However, finiteness properties of the welded braid group $WB_n$ have been studied by Zaremsky in \cite{mz}. In particular finite presentability of $WB_n'$ follows from Zaremsky's work, though explicit finite presentation for $WB_n'$ has not been obtained so far.

\medskip Bardakov, in \cite{ba1}, introduced another generalization of braid groups which he called \emph{universal braid groups} and denoted by $UB_n$. The group $UB_n$ is generated by a set of $2(n-1)$ generators: $\{ ~ \sigma_i,~ \rho_i ~ | ~ i=1, 2, \ldots, n-1 ~ \}$, where $\sigma_i$  satisfy the braid relations, and in addition there are relations \eqref{mm5}, \eqref{mm4} as defining relations. The speciality of $UB_n$ is that it is an Artin group and there are surjective homomorphisms from $UB_n$ onto all the other generalized braid groups, i.e. $GVB_n$, $SG_n$, $VB_n$ and $WB_n$. 

\medskip Since commutator subgroup is a verbal subgroup, a surjective homomorphism between two groups restricts to a surjective homomorphism between the commutator subgroups. In view of \thmref{sgnmainth}, and the fact that finite generation is quotient-closed property of groups, the following corollary follows immediately.

\begin{cor} \label{corubn}
Commutator subgroups of the universal braid groups $UB_3$ and $UB_4$ are not finitely generated. 
\end{cor} 

The following table summarizes the state of art concerning finite generation, finite presentability and perfectness of the commutator subgroups of different braid groups those are mentioned in this paper. 
\begin{table}[h!]
	\begin{center}
		\label{table1}
		\begin{tabular}{|c|c|c|c|}
			\hline
			&\textbf{Finitely generated} & \textbf{ \ Perfect \ } & \textbf{Finitely presented}\\ 
			
			\hline
			$B_n'$ & for all $n$ (\cite{gl})& iff $ ~ n \ge 5$ (\cite{gl}) & for all $n$ (\cite{gl})\\ 
			\hline
			$VB_n'$ & iff $~ n \ge 4$ (\cite{bgn1}) & iff $ ~ n \ge 5$ (\cite{bgn1}) & \textbf{?}\\
			\hline
			$WB_n'$ & for all $n$ (\cite{dg1})& iff $~ n \ge 5$ (\cite{dg1})& iff $n \ge 4$ (\cite{mz})\\
			\hline
			$GVB_n'$ & iff $~ n \ge 4$ (\thmref{gvbnmainth})& iff $~ n \ge 5$ (\thmref{gvbnmainth})& \textbf{?}\\
			\hline
			$SG_n'$ & iff $~ n \ge 5$ (\thmref{sgnmainth})& iff $n \ge 5$ (\thmref{sgnmainth}) & \textbf{?}\\
            \hline
			& Not f.g. for $n=3,4$  &&\\ $UB_n'$ & (\corref{corubn}); for $n \geq 5~$\textbf{?} & \textbf{?}  & \textbf{?}\\

			\hline
		
		\end{tabular}
	\end{center}
\end{table} 
 
\subsubsection*{Structure of the paper}  In \secref{prel}, we recall relationship between several braid groups using their group presentations and briefly describe the Reidemeister-Schreier algorithm. In \secref{gen}, we compute a set of generators for both $GVB_n'$ and $SG_n'$.  In \secref{rs}, we obtain a presentation for $GVB_n'$ using the Reidemeister-Schreier algorithm. Using this presentation, in  \secref{thm1}, we prove   \thmref{gvbnmainth}.  In \secref{sgn}, we obtain a presentation for $SG_n'$ by the Reidemeister-Schreier algorithm and use that to  prove \thmref{sgnmainth}. 

\section{Preliminaries}\label{prel}

\subsection{Relationship between $SG_n$, $GVB_n$, $B_n$, $VB_n$ and $WB_n$} Let $S_n$, $B_n$, $VB_n$ and $WB_n$ denote the symmetric group on $n$ letters, Artin's braid group on $n$ strands, virtual braid group on $n$ strands and welded braid group on $n$ strands, respectively.  Recall that the \emph{symmetric group} $S_n$ is generated by a set of $(n-1)$ generators: $\{ ~ \sigma_i ~ | ~ i=1, 2, \ldots, n-1 ~ \}$, satisfying the following set of defining relations:
$$\sigma_i^2 = 1;$$
$$\sigma_i \sigma_j=\sigma_j \sigma_i, \hbox{ if } |i-j|>1;$$
$$\sigma_i \sigma_{i+1} \sigma_i=\sigma_{i+1} \sigma_i \sigma_{i+1}.$$ 

The \emph{virtual braid group} $VB_n$ is generated by a set of $2(n-1)$ generators: $\{ ~ \sigma_i,~ \rho_i ~ | ~ i=1, 2, \ldots, n-1 ~ \}$, where $\rho_i$ satisfy the braid relations, $\sigma_i$ satisfy the symmetric group relations, and there are the following {\it mixed defining relations}: $$\sigma_i \rho_j =\rho_j \sigma_i, \hbox{ if } |i-j|>1;$$
	$$\rho_{i+1} \sigma_i \sigma_{i+1} = \sigma_i \sigma_{i+1} \rho_i.$$
	
The \emph{welded braid group} $WB_n$ is a quotient of $VB_n$ obtained by including the following defining relations in the above presentation for $VB_n$: 
	$$\sigma_i \rho_{i+1} \rho_i = \rho_{i+1} \rho_i \sigma_{i+1}.$$

\medskip 

The relationship between the groups $~ S_n$, $~ B_n$, $~ VB_n$, $~ WB_n $, $~ GVB_n ~$, $ ~ SG_n ~ $ and $UB_n$ is given by the following diagram of surjective homomorphisms:
\begin{equation}\label{gvbndiagram}
	\vcenter{
		\xymatrix@R=2ex{
			& {GVB_n} \ar[dd]_{\gamma} \ar[rdd]^{\alpha} & {UB_n} \ar[r]^{\kappa} \ar[l]_{\xi} & {SG_n} \ar[ldd]^{\omega} \\
			&&&& \\
			{WB_n} & {VB_n} \ar[l]_{\zeta} \ar[dd]_{\delta}
			&{B_n} \ar[ldd]^{\beta} &\\
			&&&& \\
			& {S_n} &
		}
	}
\end{equation}

\noindent where $\alpha$, $\beta$, $\gamma$, $\delta$, $\omega ~$ are quotient maps defined by the normal subgroups generated by the words $\{\rho_i\}$, $\{\sigma_i^2\}$, $\{\sigma_i^2\}$, $\{\rho_i\}$, $\{\rho_i\}$ respectively; $~ \zeta$ is the quotient map defined by the normal subgroup generated by the forbidden relators; $~ \xi$ is the quotient map defined by the normal subgroup generated by the relators in \eqref{mm2}, \eqref{mm3}, and the relators $ \{ \rho_i \rho_{i+1} \rho_i \rho_{i+1}^{-1} \rho_i^{-1} \rho_{i+1}^{-1} \} $; $~ \kappa$ is the quotient map defined by the normal subgroup generated by the relators in \eqref{mm2}, \eqref{mm3}, and the relators $ \{ \sigma_i \rho_i \sigma_i^{-1} \rho_i^{-1} \}.$

Note that the map $\gamma : VB_n \rightarrow S_n$ defined above is different from the map from $VB_n$ to $S_n$ considered while defining {\em pure virtual braid groups}, e.g. in \cite[Section 5]{bb}.

\subsection{The Reidemeister-Schreier algorithm}\label{rsa}

\begin{definition}
	Let $G$ be a group and $H$ be a subgroup of $G$. Suppose that, $ \langle ~ S ~ | ~ R ~ \rangle $ be a presentation for $G$. A set $\Lambda$ consisting of some words in the generators from $S$ is called a \textit{Schreier set of coset representatives} for $H$ in $G$ if \\
	(i) every right coset of $H$ in $G$ contains exactly one word from $\Lambda$, and\\
	(ii) for each word in $\Lambda$ any initial segment of that word is also in $\Lambda$.
\end{definition}

\begin{theorem}[Schreier, 1927]
	For every $H \le G$, a Schreier set of coset representatives for $H$ in $G$ exists.
\end{theorem}

We describe the Reidemeister-Schreier method as an algorithm, which provides an effective way to deduce a presentation for a subgroup of a group with a given presentation. Let $H \le G$, and $ \langle ~ S ~ | ~ R ~ \rangle $ be a presentation for $G$. In order to deduce a presentation for $H$ we proceed as follows.

\underline{Step 1}: Find $\Lambda$, a Schreier set of coset representatives for $H$ in $G$. One may use the Todd-Coxeter algorithm to find such a Schreier set,{\tiny } if $[G:H] < \infty$.\\

\underline{Step 2}: Deduce a set of generators $\{ ~ S_{\lambda, a} ~ | ~ \lambda \in \Lambda, ~ a \in S ~ \}$ for $H$ defined by:
$$S_{\lambda, a}=(\lambda a) (\overline{\lambda a})^{-1},$$
where for any $x \in G, ~$ $\overline{x} \in \Lambda$ denotes the unique element in $\Lambda \cap Hx$.\\

\underline{Step 3}: If for some $\lambda \in \Lambda$ and $a \in S$, the words $\lambda a$ and $\overline{\lambda a}$ are freely equal, then collect the relations $S_{\lambda, a}=1$ as part of the set of defining relations for $H$.\\

\underline{Step 4}: Compute the remaining defining relators $~ \tau (\lambda ~ r_{\mu} ~ \lambda^{-1}) ~$ for $H$,
for all $\lambda \in \Lambda$, and all the defining relators $r_{\mu} \in R$, where $\tau$, called the \emph{rewriting process}, is defined as follows. For a word $a_{i_1}^{\epsilon_1} \dots a_{i_p}^{\epsilon_p}$ in the generators from $S$, with $\epsilon_j = 1$ or $-1 ~$ for $1 \le j \le p$,
$$\tau( ~ a_{i_1}^{\epsilon_1} ~ \dots ~ a_{i_p}^{\epsilon_p} ~ ) ~ := ~ S_{K_{i_1},a_{i_1}}^{\epsilon_1} ~ \dots ~ S_{K_{i_p},a_{i_p}}^{\epsilon_p}, $$

$$ \hbox{ where } ~ K_{i_j} =  \begin{cases} 
\overline{a_{i_1}^{\epsilon_1} \dots a_{i_{j-1}}^{\epsilon_{j-1}}} & \hbox{ if } \epsilon_j = 1 , \\

\hbox{   } \\

\overline{a_{i_1}^{\epsilon_1} \dots a_{i_j}^{\epsilon_j}} & \hbox{ if } \epsilon_j = -1 .
\end{cases}$$\\
The above algorithm will produce a presentation $ \langle ~ \overline{ S } ~ | ~ \overline{ R } ~ \rangle $ for $H$ where we have $\overline{ S } = \{  ~ S_{\lambda, a} ~ | ~ \lambda \in \Lambda, ~ a \in S ~ \}$ and $\overline{ R } = \{ ~ S_{\lambda_t, a_t} ~ | ~ \lambda_t \in \Lambda,~ a_t \in S,~ \lambda_t a_t \hbox{ freely equal to } \overline{ \lambda_t a_t } ~ \} ~ \cup $ $~ \{ ~ \tau(\lambda ~ r_{\mu} ~ \lambda^{-1}) ~ | ~ \lambda \in \Lambda,~ r_{\mu} \in R ~ \} $.\\

Refer \cite[Theorem 2.9]{mks} for detailed explanation and proofs.

\section{ A set of generators for $GVB_n'$ and $SG_n'$} \label{gen} 
Let $G_n=GVB_n$ or $SG_n$. In order to deduce a set of generators for $G_n'$, we will execute Step 1 and Step 2 of the Reidemeister-Schreier algorithm as described in \secref{rsa}.

Define the map $\phi$: 
\begin{equation*}1 \xrightarrow {} G_n' \xrightarrow{} G_n \xrightarrow{\phi} \Z \times \Z \xrightarrow{} 1\end{equation*}
where, for $i=1, \ldots, n-1$, $\phi(\sigma_i)=\widetilde{\sigma_1}$ , $\phi(\rho_i)=\widetilde{\rho_1}$; here  $\widetilde{\sigma_1}$ and $\widetilde{\rho_1}$ are the generators of the 2 copies of $\Z$. Here, Image($\phi$) is isomorphic to the abelianization of $G_n$, denoted as $G_n^{ab}$. To verify this, we abelianize the above presentation for $G_n$ by inserting the relations $ ~ xy=yx ~ $ in the presentation for all $x,y \in \{ ~ \sigma_i, \rho_i ~ | ~ 1 \le i \le n-1 \ \} $. The resulting presentation is the following:
$$ G_n^{ab} = ~ < \sigma_1, \rho_1 ~ | ~ \sigma_1 \rho_1 = \rho_1 \sigma_1 ~ >$$
Clearly, $G_n^{ab}$ is isomorphic to $\Z \times \Z$. But as $\phi$ is onto, Image($\phi$) = $\Z \times \Z$.  
Hence, Image($\phi$) is isomorphic to $G_n^{ab}$.  Hence, $\phi$ defines the above short exact sequence.\\

We have the following lemma.

\begin{lemma}\label{gvbnlemma1}
 $G_n'$ is generated by the words $~ ~ ~ \alpha_{m, k, i}=\sigma_1^m \rho_1^k \sigma_i \rho_1^{-k} \sigma_1^{-1} \sigma_1^{-m} ~ ~ ~ $ and \\ $\beta_{m, k, i}=\sigma_1^m \rho_1^k \rho_i \rho_1^{-k} \rho_1^{-1} \sigma_1^{-m}, ~ ~ ~ $ 
where $m, k \in \Z$, $~ ~ ~ 1 \le i \le n-1$.
\end{lemma}

\begin{proof} Consider the Schreier set of coset representatives for $G_n'$ in $G_n$:
$$\Lambda=\{ \sigma_1^m \rho_1^k \ | \ m, k \in \Z \}.$$

Following the Reidemeister-Schreier algorithm, the group $G_n'$ is generated by the set:
$$\{S_{\lambda, a}=(\lambda a) (\overline{\lambda a})^{-1} \ | \ \lambda \in \Lambda, \ a \in \{\sigma_i, \rho_i | \ i=1, 2, \ldots, n-1\} \}.$$
Choose $\lambda=\sigma_1^m \rho_1^k$ from $\Lambda$.  
For  $a=\sigma_i$, $S_{\lambda, a} = \sigma_1^m \rho_1^k \sigma_i \rho_1^{-k} \sigma_1^{-1} \sigma_1^{-m}$. For $a=\rho_i$, $S_{\lambda, a} = \sigma_1^m \rho_1^k \rho_i \rho_1^{-k} \rho_1^{-1} \sigma_1^{-m}$. Hence, $G_n'$ is generated by the following elements:
\begin{equation*}\alpha_{m, k, i}=S_{\sigma_1^m \rho_1^k, \sigma_i}=\sigma_1^m \rho_1^k \sigma_i \rho_1^{-k} \sigma_1^{-1} \sigma_1^{-m},\end{equation*}
\begin{equation*}\beta_{m, k, i}=S_{\sigma_1^m \rho_1^k, \rho_i}=\sigma_1^m \rho_1^k \rho_i \rho_1^{-k} \rho_1^{-1} \sigma_1^{-m},\end{equation*}
where $m, k \in \Z$, $1 \le i \le n-1$.\end{proof}

 \section{Generators and Defining Relations for $GVB_n'$} \label{rs} 

In order to deduce a set of defining relations for $GVB_n'$, we will execute Step 3 and Step 4 of the Reidemeister-Schreier algorithm as described in \secref{rsa}. We have the following lemma.

\begin{lemma}\label{gvbnlemma2}
	$GVB_n'$ has the following presentation.
	
	Set of generators: $$ \{ ~ \alpha_{m,k,1},~ \alpha_{m,k,2},~ \beta_{m,k,2},~ \alpha_j,~ \beta_{m,j} ~ | ~ m, k \in \Z,~ 3 \le j \le n-1 ~ \} $$
	
	Set of defining relations:
	\medskip for all $~ m, k \in \Z$,
	\begin{equation}\label{gvbn0}
		\alpha_{m,0,1}=1;
	\end{equation}
	\begin{equation}\label{gvbn1} \alpha_{m,k,1} ~ \alpha_{j} ~ \alpha_{m+1,k,1}^{-1} ~ \alpha_{j}^{-1} ~ = 1, ~ j \ge 3 ; \end{equation}
	\begin{equation}\label{gvbn2} \alpha_{m,k,2} ~ \alpha_{j} ~ \alpha_{m+1,k,2}^{-1} ~ \alpha_{j}^{-1} ~ =1, ~ j \ge 4 ; \end{equation}
	\begin{equation}\label{gvbn3} \alpha_{i} ~ \alpha_{j} ~ \alpha_{i}^{-1} ~ \alpha_{j}^{-1} ~   =1, ~ i,j \ge 3, ~ |i-j| > 1 ; \end{equation}
	\begin{equation}\label{gvbn4} \beta_{m,k,2} ~ \beta_{m,j} ~ \beta_{m,k+1,2}^{-1} ~ \beta_{m,j}^{-1} ~  =1, ~ j \ge 4 ; \end{equation}
	\begin{equation}\label{gvbn5} \beta_{m,i} ~ \beta_{m,j} ~ \beta_{m,i}^{-1} ~ \beta_{m,j}^{-1} ~ =1, ~ i,j \ge 3, ~ |i-j| > 1 ; \end{equation}
	\begin{equation}\label{gvbn6} \alpha_{m,k,1} ~ \beta_{m+1,j} ~ \alpha_{m,k+1,1}^{-1} ~ \beta_{m,j}^{-1} ~ =1, ~ j \ge 3 ; \end{equation}
	\begin{equation}\label{gvbn7} \alpha_{m,k,2} ~ \beta_{m+1,j} ~ \alpha_{m,k+1,2}^{-1} ~ \beta_{m,j}^{-1} ~  =1, ~ j \ge 4 ; \end{equation}
	\begin{equation}\label{gvbn8} \alpha_{i} ~ \beta_{m+1,k,2} ~ \alpha_{i}^{-1} ~ \beta_{m,k,2}^{-1} ~ =1, ~ i \ge 4 ; \end{equation}
	\begin{equation}\label{gvbn9} \alpha_{i} ~ \beta_{m+1,j} ~ \alpha_{i}^{-1} ~ \beta_{m,j}^{-1} ~ =1, ~ i,j \ge 3, ~ |i-j|>1 ; \end{equation}
	\begin{equation}\label{gvbn10} \alpha_{m,k,1} ~ \alpha_{m+1,k,2} ~ \alpha_{m+2,k,1} ~ \alpha_{m+2,k,2}^{-1} ~ \alpha_{m+1,k,1}^{-1} ~ \alpha_{m,k,2}^{-1} ~ =1 ; \end{equation}
	\begin{equation}\label{gvbn11} \alpha_{m,k,2} ~ \alpha_{3} ~ \alpha_{m+2,k,2} ~ \alpha_{3}^{-1} ~ \alpha_{m+1,k,2}^{-1} ~ \alpha_{3}^{-1} ~ =1 ; \end{equation}
	\begin{equation}\label{gvbn12} \alpha_{i} ~ \alpha_{i+1} ~ \alpha_{i} ~ \alpha_{i+1}^{-1} ~ \alpha_{i}^{-1} ~ \alpha_{i+1}^{-1}  =1, ~ i \ge 3 ; \end{equation}
	\begin{equation}\label{gvbn13} \beta_{m,k+1,2} ~ \beta_{m,k+2,2}^{-1} ~ \beta_{m,k,2}^{-1} ~ =1 ; \end{equation}
	\begin{equation}\label{gvbn14} \beta_{m,k,2} ~ \beta_{m,3} ~ \beta_{m,k+2,2} ~ \beta_{m,3}^{-1} ~ \beta_{m,k+1,2}^{-1} ~ \beta_{m,3}^{-1} ~  =1 ; \end{equation}
	\begin{equation}\label{gvbn15} \beta_{m,i} ~ \beta_{m,i+1} ~ \beta_{m,i} ~ \beta_{m,i+1}^{-1} ~ \beta_{m,i}^{-1} ~ \beta_{m,i+1}^{-1} ~  =1, ~ i \ge 3 ; \end{equation}
	\begin{equation}\label{gvbn16} \alpha_{m,k+1,2} ~ \alpha_{m+1,k+1,1} ~ \beta_{m+2,k,2}^{-1} ~ \alpha_{m+1,k,1}^{-1} ~ \alpha_{m,k,2}^{-1} ~ =1 ; \end{equation}
	\begin{equation}\label{gvbn17} \beta_{m,k,2} ~ \alpha_{3} ~ \alpha_{m+1,k+1,2} ~ \beta_{m+2,3}^{-1} ~ \alpha_{m+1,k,2}^{-1} ~ \alpha_{3}^{-1} ~  =1 ; \end{equation}
	\begin{equation}\label{gvbn18} \beta_{m,i} ~ \alpha_{i+1} ~ \alpha_{i} ~ \beta_{m+2,i+1}^{-1} ~ \alpha_{i}^{-1} ~ \alpha_{i+1}^{-1} ~ =1, ~ i \ge 3 ; \end{equation}
	\begin{equation}\label{gvbn19} \beta_{m,k,2} ~ \alpha_{m,k+1,1} ~ \alpha_{m+1,k+1,2} ~ \alpha_{m+1,k,2}^{-1} ~ \alpha_{m,k,1}^{-1} ~ =1 ; \end{equation}
	\begin{equation}\label{gvbn20} \beta_{m,3} ~ \alpha_{m,k+1,2} ~ \alpha_{3} ~ \beta_{m+2,k,2}^{-1} ~ \alpha_{3}^{-1} ~ \alpha_{m,k,2}^{-1} ~ =1 ; \end{equation}
	\begin{equation}\label{gvbn21} \beta_{m,i+1} ~ \alpha_{i} ~ \alpha_{i+1} ~ \beta_{m+2,i}^{-1} ~ \alpha_{i+1}^{-1} ~ \alpha_{i}^{-1} ~ =1, ~ i \ge 3. \end{equation}

\end{lemma}
\medskip

\begin{proof}
	We execute Step 3 of the Reidemeister-Schreier algorithm. We consider all the pairs $~ (\lambda, a) ~$ with $~ \lambda \in \Lambda=\{ \sigma_1^m \rho_1^k \ | \ m, k \in \Z \},~ a \in S= \{\sigma_i, \rho_i ~ | \ i=1, 2, \ldots, n-1\}$. We need to find the pairs $~ (\lambda, a) ~$ for which $\lambda a$ and $\overline{\lambda a}$ are freely equal.
	
	Note that, $\sigma_1^m \rho_1^k ~ \sigma_i$ and $\overline{\sigma_1^m \rho_1^k ~ \sigma_i}=\sigma_1^{m+1} \rho_1^k$ are freely equal if and only if $k=0$ and $i=1$. Also note that, $\sigma_1^m \rho_1^k ~ \rho_i$ and $\overline{\sigma_1^m \rho_1^k ~ \rho_i}=\sigma_1^m \rho_1^{k+1}$ are freely equal if and only if $i=1$. Hence, we get the following relations as some of the defining relations for $GVB_n'$.
	$$S_{\sigma_1^m, \sigma_1}=1, ~ \hbox{ i.e. } ~ \alpha_{m,0,1}=1;$$
	$$S_{\sigma_1^m \rho_1^k, \rho_1}=1, ~ \hbox{ i.e. } ~ \beta_{m,k,1}=1.$$
	
	\medskip Now, following Step 4 of the Reidemeister-Schreier algorithm, we calculate the terms $\tau(\lambda r_{\mu} \lambda^{-1})$ for each $\lambda \in \Lambda$ and for each of the defining relations $r_{\mu}=1$ of $GVB_n$, as follows:
		
	$$r_1 = \sigma_i \sigma_j \sigma_i^{-1} \sigma_j^{-1} = 1, \ |i-j|>1;$$
	$$r_2 = \rho_i \rho_j \rho_i^{-1} \rho_j^{-1} = 1, \ |i-j|>1;$$
	$$r_3 = \sigma_i \rho_j \sigma_i^{-1} \rho_j^{-1} = 1, \ |i-j|>1;$$
	$$r_4 = \sigma_i \sigma_{i+1} \sigma_i \sigma_{i+1}^{-1} \sigma_i^{-1} \sigma_{i+1}^{-1}=1;$$
	$$r_5 = \rho_i \rho_{i+1} \rho_i \rho_{i+1}^{-1} \rho_i^{-1} \rho_{i+1}^{-1} = 1;$$
	$$r_6 = \rho_i \sigma_{i+1} \sigma_i \rho_{i+1}^{-1} \sigma_i^{-1} \sigma_{i+1}^{-1} = 1;$$
	$$r_7 = \rho_{i+1} \sigma_i \sigma_{i+1} \rho_i^{-1} \sigma_{i+1}^{-1} \sigma_i^{-1} = 1.$$\\
	Choose any element $\lambda = \sigma_1^m \rho_1^k \in \Lambda.$ We deduce the following.
	
	$$\tau(\lambda r_1 \lambda^{-1}) = S_{\sigma_1^{m} \rho_1^{k}, \sigma_i} ~ S_{\sigma_1^{m+1} \rho_1^{k}, \sigma_j} ~ S_{\sigma_1^{m+1} \rho_1^{k}, \sigma_i}^{-1} ~ S_{\sigma_1^{m} \rho_1^{k}, \sigma_j}^{-1};$$
	
	$$\tau(\lambda r_2 \lambda^{-1}) = S_{\sigma_1^{m} \rho_1^{k}, \rho_i} ~ S_{\sigma_1^{m} \rho_1^{k+1}, \rho_j} ~ S_{\sigma_1^{m} \rho_1^{k+1}, \rho_i}^{-1} ~ S_{\sigma_1^{m} \rho_1^{k}, \rho_j}^{-1};$$
	
	$$\tau(\lambda r_3 \lambda^{-1}) = S_{\sigma_1^{m} \rho_1^{k}, \sigma_i} ~ S_{\sigma_1^{m+1} \rho_1^{k}, \rho_j} ~ S_{\sigma_1^{m} \rho_1^{k+1}, \sigma_i}^{-1} ~ S_{\sigma_1^{m} \rho_1^{k}, \rho_j}^{-1};$$
	
	$$\tau(\lambda r_4 \lambda^{-1}) = S_{\sigma_1^{m} \rho_1^{k}, \sigma_i} ~ S_{\sigma_1^{m+1} \rho_1^{k}, \sigma_{i+1}} ~ S_{\sigma_1^{m+2} \rho_1^{k}, \sigma_i} ~ S_{\sigma_1^{m+2} \rho_1^{k}, \sigma_{i+1}}^{-1} ~ S_{\sigma_1^{m+1} \rho_1^{k}, \sigma_i}^{-1} ~ S_{\sigma_1^{m} \rho_1^{k}, \sigma_{i+1}}^{-1};$$
	
	$$\tau(\lambda r_5 \lambda^{-1}) = S_{\sigma_1^{m} \rho_1^{k}, \rho_i} ~ S_{\sigma_1^{m} \rho_1^{k+1}, \rho_{i+1}} ~ S_{\sigma_1^{m} \rho_1^{k+2}, \rho_i} ~ S_{\sigma_1^{m} \rho_1^{k+2}, \rho_{i+1}}^{-1} ~ S_{\sigma_1^{m} \rho_1^{k+1}, \rho_i}^{-1} ~ S_{\sigma_1^{m} \rho_1^{k}, \rho_{i+1}}^{-1};$$
	
	$$\tau(\lambda r_6 \lambda^{-1}) = S_{\sigma_1^{m} \rho_1^{k}, \rho_i} ~ S_{\sigma_1^{m} \rho_1^{k+1}, \sigma_{i+1}} ~ S_{\sigma_1^{m+1} \rho_1^{k+1}, \sigma_i} ~ S_{\sigma_1^{m+2} \rho_1^{k}, \rho_{i+1}}^{-1} ~ S_{\sigma_1^{m+1} \rho_1^{k}, \sigma_i}^{-1} ~ S_{\sigma_1^{m} \rho_1^{k}, \sigma_{i+1}}^{-1};$$
	
	$$\tau(\lambda r_7 \lambda^{-1}) = S_{\sigma_1^{m} \rho_1^{k}, \rho_{i+1}} ~ S_{\sigma_1^{m} \rho_1^{k+1}, \sigma_i} ~ S_{\sigma_1^{m+1} \rho_1^{k+1}, \sigma_{i+1}} ~ S_{\sigma_1^{m+2} \rho_1^{k}, \rho_i}^{-1} ~ S_{\sigma_1^{m+1} \rho_1^{k}, \sigma_{i+1}}^{-1} ~ S_{\sigma_1^{m} \rho_1^{k}, \sigma_i}^{-1}.$$
	
	Hence we get the following set of defining relations for $GVB_n'$.
	\begin{equation}\label{eq1}
		\alpha_{m,0,1}=1;
	\end{equation}
	\begin{equation}\label{eq2}
		\beta_{m,k,1}=1;
	\end{equation}
	\begin{equation}\label{eq3}
		\alpha_{m,k,i} ~ \alpha_{m+1,k,j} ~ \alpha_{m+1,k,i}^{-1} ~ \alpha_{m,k,j}^{-1} = 1,~ |i-j|>1;
	\end{equation}
	\begin{equation}\label{eq4}
		\beta_{m,k,i} ~ \beta_{m,k+1,j} ~ \beta_{m,k+1,i}^{-1} ~ \beta_{m,k,j}^{-1} = 1,~ |i-j|>1;
	\end{equation}
	\begin{equation}\label{eq5}
		\alpha_{m,k,i} ~ \beta_{m+1,k,j} ~ \alpha_{m,k+1,i}^{-1} ~ \beta_{m,k,j}^{-1} = 1,~ |i-j|>1;
	\end{equation}
	\begin{equation}\label{eq6}
		\alpha_{m,k,i} ~ \alpha_{m+1,k,i+1} ~ \alpha_{m+2,k,i} ~ \alpha_{m+2,k,i+1}^{-1} ~ \alpha_{m+1,k,i}^{-1} ~ \alpha_{m,k,i+1}^{-1} = 1;
	\end{equation}
	\begin{equation}\label{eq7}
		\beta_{m,k,i} ~ \beta_{m,k+1,i+1} ~ \beta_{m,k+2,i} ~ \beta_{m,k+2,i+1}^{-1} ~ \beta_{m,k+1,i}^{-1} ~ \beta_{m,k,i+1}^{-1} = 1;
	\end{equation}
	\begin{equation}\label{eq8}
		\beta_{m,k,i} ~ \alpha_{m,k+1,i+1} ~ \alpha_{m+1,k+1,i} ~ \beta_{m+2,k,i+1}^{-1} ~ \alpha_{m+1,k,i}^{-1} ~ \alpha_{m,k,i+1}^{-1} = 1;
	\end{equation}
	\begin{equation}\label{eq9}
		\beta_{m,k,i+1} ~ \alpha_{m,k+1,i} ~ \alpha_{m+1,k+1,i+1} ~ \beta_{m+2,k,i}^{-1} ~ \alpha_{m+1,k,i+1}^{-1} ~ \alpha_{m,k,i}^{-1} = 1.
	\end{equation}
	
	\medskip For $i=1, ~ j \ge 3$, \eqref{eq4} gives the relations:
	$$\beta_{m,k,1} ~ \beta_{m,k+1,j} ~ \beta_{m,k+1,1}^{-1} ~ \beta_{m,k,j}^{-1} = 1, ~ j \ge 3, ~ m, k \in \Z.$$
	
	Using \eqref{eq2} and the above relations we get:
	\begin{equation*}
	\beta_{m,k+1,j} = \beta_{m,k,j}, ~ j \ge 3, ~ m,k \in \Z.
	\end{equation*}
	
	Iterating the above relations we have:
	\begin{equation}\label{eq10}
		\beta_{m,k,j} = \beta_{m,0,j},~ j \ge 3, ~ m,k \in \Z.
	\end{equation}
	
	\medskip For $j=1, ~ i \ge 3$, \eqref{eq5} gives the relations:
	$$\alpha_{m,k,i} ~ \beta_{m+1,k,1} ~ \alpha_{m,k+1,i}^{-1} ~ \beta_{m,k,1}^{-1} = 1, ~ i \ge 3, ~ m, k \in \Z.$$
	
	Using \eqref{eq2} and the above relations we get:
	\begin{equation*}
	\alpha_{m,k+1,i} = \alpha_{m,k,i}, ~ i \ge 3, ~ m,k \in \Z.
	\end{equation*}
	
	Iterating the above relations we have:
	\begin{equation}\label{eq11}
	\alpha_{m,k,i} = \alpha_{m,0,i},~ i \ge 3, ~ m,k \in \Z.
	\end{equation}
	
	For $k=0, ~ i=1$, \eqref{eq3} gives the relations:
	$$\alpha_{m,0,1} ~ \alpha_{m+1,0,j} ~ \alpha_{m+1,0,1}^{-1} ~ \alpha_{m,0,j}^{-1} = 1,~ j \ge 3, ~ m \in \Z.$$
	
	Using \eqref{eq1} and the above relations we have:
	$$\alpha_{m+1,0,j} = \alpha_{m,0,j}, ~ j \ge 3, ~ m \in \Z.$$
	
	Iterating the above relations we get:
	\begin{equation}\label{eq12}
		\alpha_{m,0,j} = \alpha_{0,0,j},~ j \ge 3,~ m \in \Z.
	\end{equation}
	
	\eqref{eq11} and \eqref{eq12} together gives:
	\begin{equation}\label{eq13}
		\alpha_{m,k,j} = \alpha_{0,0,j},~ j \ge 3,~ m,k \in \Z.
	\end{equation}
	
	Using \eqref{eq13} we replace $\alpha_{m,k,j}$ by $\alpha_{0,0,j}$ for all $(m, k) \ne (0,0)$ and $j \ge 3$, in all the defining relations, and remove all $\alpha_{m,k,j}$ with $(m, k) \ne (0,0),~ j \ge 3$, from the set of generators. After this replacement we denote $\alpha_{0,0,j}$ simply by $\alpha_j$ for $j \ge 3$.\\
	
	Using \eqref{eq10} we replace $\beta_{m,k,j}$ by $\beta_{m,0,j}$ for all $k \ne 0,~ j \ge 3,~ m \in \Z,$ in all the defining relations, and remove all $\beta_{m,k,j}$ with $k \ne 0, ~ j \ge 3, ~ m \in \Z,$ from the set of generators. After this replacement we denote $\beta_{m,0,j}$ simply by $\beta_{m,j}$ for $j \ge 3, ~ m \in \Z.$\\
	
	Lastly, using \eqref{eq2} we replace $\beta_{m,k,1}$ by 1, for all $m,k \in \Z$, in all the defining relations, and remove all $\beta_{m,k,1}$ from the set of generators.
	
	After incorporating all the replacements and notation changes as described above we get the following set of generators for $GVB_n'$:
	$$\{ ~ \alpha_{m,k,1},~ \alpha_{m,k,2},~ \beta_{m,k,2},~ \alpha_j,~ \beta_{m,j} ~ | ~ m, k \in \Z,~ 3 \le j \le n-1 ~ \},$$
	
	and the following list of defining relations for $GVB_n'$.\\
	
	To start with, we have the relations \eqref{gvbn0}. (same as \eqref{eq1})\\
		
	For \eqref{eq3} we have the following 3 possible cases:\\
	Case 1: $i=1, ~ j \ge 3$; gives the relations: \eqref{gvbn1}.\\
	Case 2: $i=2, ~ j \ge 4$; gives the relations: \eqref{gvbn2}.\\
	Case 3: $i,j \ge 3, ~ |i-j|>1$; gives the relations: \eqref{gvbn3}.\\

	For \eqref{eq4} we have the following 3 possible cases:\\
	Case 1: $i=1, ~ j \ge 3$; gives no nontrivial relation.\\
	Case 2: $i=2, ~ j \ge 4$; gives the relations: \eqnref{gvbn4}.\\
	Case 3: $i,j \ge 3, ~ |i-j|>1$; gives the relations: \eqnref{gvbn5}.\\

	For \eqref{eq5} we have the following 5 possible cases:\\
	Case 1: $i=1, ~ j \ge 3$; gives the relations: \eqnref{gvbn6}.\\
	Case 2: $j=1, ~ i \ge 3$; gives no nontrivial relation.\\
	Case 3: $i=2, ~ j \ge 4$; gives the relations: \eqnref{gvbn7}.\\
	Case 4: $j=2, ~ i \ge 4$; gives the relations: \eqnref{gvbn8}.\\
	Case 5: $i,j \ge 3, ~ |i-j|>1$; gives the relations: \eqnref{gvbn9}.\\

	For \eqref{eq6} we have the following 3 possible cases:\\
	Case 1: $i=1$; gives the relations: \eqnref{gvbn10}.\\
	Case 2: $i=2$; gives the relations: \eqnref{gvbn11}.\\
	Case 3: $i \ge 3$; gives the relations: \eqnref{gvbn12}.\\

	For \eqref{eq7} we have the following 3 possible cases:\\
	Case 1: $i=1$; gives the relations: \eqnref{gvbn13}.\\
	Case 2: $i=2$; gives the relations: \eqnref{gvbn14}.\\
	Case 3: $i \ge 3$; gives the relations: \eqnref{gvbn15}.\\

	For \eqref{eq8} we have the following 3 possible cases:\\
	Case 1: $i=1$; gives the relations: \eqnref{gvbn16}.\\
	Case 2: $i=2$; gives the relations: \eqnref{gvbn17}.\\
	Case 3: $i \ge 3$; gives the relations: \eqnref{gvbn18}.\\

	For \eqref{eq9} we have the following 3 possible cases:\\
	Case 1: $i=1$; gives the relations: \eqnref{gvbn19}.\\
	Case 2: $i=2$; gives the relations: \eqnref{gvbn20}.\\
	Case 3: $i \ge 3$; gives the relations: \eqnref{gvbn21}.\\
	
	This completes the proof of the lemma.
\end{proof}

\section{Proof of \thmref{gvbnmainth}}\label{thm1} 

\subsection{Infinite generation of $GVB_3'$}

From \lemref{gvbnlemma2} we have the following presentation for $GVB_3'$:\\

Generators:  $~ \{ ~ \alpha_{m,k,1}, ~ \alpha_{m,k,2}, ~ \beta_{m,k,2} ~ | ~ m, k \in \Z ~ \}$.\\

Defining Relations: for all $m, k \in \Z,$
\begin{equation}\label{gvbn22} \alpha_{m,k,1} ~ \alpha_{m+1,k,2} ~ \alpha_{m+2,k,1} ~ \alpha_{m+2,k,2}^{-1} ~ \alpha_{m+1,k,1}^{-1} ~ \alpha_{m,k,2}^{-1} ~ =1 ; \end{equation}
\begin{equation}\label{gvbn23} \beta_{m,k+1,2} ~ \beta_{m,k+2,2}^{-1} ~ \beta_{m,k,2}^{-1} ~ =1 ; \end{equation}
\begin{equation}\label{gvbn24} \alpha_{m,k+1,2} ~ \alpha_{m+1,k+1,1} ~ \beta_{m+2,k,2}^{-1} ~ \alpha_{m+1,k,1}^{-1} ~ \alpha_{m,k,2}^{-1} ~ =1 ; \end{equation}
\begin{equation}\label{gvbn25} \beta_{m,k,2} ~ \alpha_{m,k+1,1} ~ \alpha_{m+1,k+1,2} ~ \alpha_{m+1,k,2}^{-1} ~ \alpha_{m,k,1}^{-1} ~ =1 ; \end{equation}
\begin{equation}\label{gvbn26}
	\alpha_{m,0,1} = 1.
\end{equation}

We have the following lemma.

\begin{lemma}\label{gvbnlemma3}
	$GVB_3'$ is not finitely generated.
\end{lemma}

\begin{proof}
	From \eqnref{gvbn25} we have:
	$$\beta_{m,k,2} ~ = ~ \alpha_{m,k,1} ~ \alpha_{m+1,k,2} ~ \alpha_{m+1,k+1,2}^{-1} ~ \alpha_{m,k+1,1}^{-1}.$$
	
	Replacing these values of $\beta_{m,k,2}$ in the other relations and removing the generators $\beta_{m,k,2}$ from the set of generators we obtain an equivalent presentation for $GVB_3'$ as follows:\\
	
	Generators: $~ \{ ~ \alpha_{m,k,1}, ~ \alpha_{m,k,2} ~ | ~ m, k \in \Z ~ \}; $\\
	
	Defining relations: For all $m, k \in \Z,$
	\begin{equation}\label{gvbn27} \alpha_{m,k,1} ~ \alpha_{m+1,k,2} ~ \alpha_{m+2,k,1} ~ \alpha_{m+2,k,2}^{-1} ~ \alpha_{m+1,k,1}^{-1} ~ \alpha_{m,k,2}^{-1} ~ =1 ; \end{equation}
	\begin{equation}\label{gvbn28} \alpha_{m,k+1,1} ~ \alpha_{m+1,k+1,2} ~ \alpha_{m+1,k+2,2}^{-1} ~ \alpha_{m,k+2,1}^{-1} ~ = \end{equation}
	\begin{equation*}\label{gvbn28'}
		~ \alpha_{m,k,1} ~ \alpha_{m+1,k,2} ~ \alpha_{m+1,k+1,2}^{-1} ~ \alpha_{m,k+1,1}^{-1} ~ \alpha_{m,k+2,1} ~ \alpha_{m+1,k+2,2} ~ \alpha_{m+1,k+3,2}^{-1} ~ \alpha_{m,k+3,1}^{-1} ; \end{equation*}
	\begin{equation}\label{gvbn29}
		\alpha_{m,k+1,2} ~ \alpha_{m+1,k+1,1} ~ \alpha_{m+2,k+1,1} ~ \alpha_{m+3,k+1,2} ~ \alpha_{m+3,k,2}^{-1} ~ \alpha_{m+2,k,1}^{-1} ~ \alpha_{m+1,k,1}^{-1} ~ \alpha_{m,k,2}^{-1} = 1;
	\end{equation}
	\begin{equation}\label{gvbn30}
		\alpha_{m,0,1}=1.
	\end{equation}
	
	Now, consider the words $w_{m,k} = \alpha_{m, k, 1} ~ \alpha_{m+1, k, 2}$ in $GVB_3'$ for all $m,k \in \Z$. Now let's construct the quotient group $ ~ GVB_3' / W  , ~ $ where $W = \overline{ \langle w_{m,k} ~ | ~ m, k \in \Z \rangle } $, the normal subgroup generated by the words $w_{m,k}$. We obtain a presentation for $ ~ GVB_3' / W  ~ $ by inserting the relations $\alpha_{m, k, 1} ~ \alpha_{m+1, k, 2} = 1, ~ m, k \in \Z$ in the last presentation for $GVB_3'$; and the obtained presentation for $ ~ GVB_3' / W  ~ $ is as follows:\\
	
	Generators: $~ \{ ~ \alpha_{m,k,1} ~ | ~ m, k \in \Z ~ \}; $\\
	
	Defining relations:
	\begin{equation}\label{gvbn31}
		\alpha_{m+2,k,1} = \alpha_{m-1,k,1}^{-1};
	\end{equation}
	\begin{equation}\label{gvbn32}
		\alpha_{m-1,k+1,1}^{-1} ~ \alpha_{m+1,k+1,1} ~ \alpha_{m+1,k,1}^{-1} ~ \alpha_{m-1,k,1} = 1;
	\end{equation}
	\begin{equation}\label{gvbn33}
		\alpha_{m,0,1} = 1.
	\end{equation}
	
	Now, we consider the words $v_{m,k} = \alpha_{m+1,k,1}^{-1} ~ \alpha_{m-1,k,1}$ in $GVB_3' / W$ for all $m, k \in \Z.$ And we consider the quotient group $ ~ ( GVB_3' / W ) / V,$ where $V = \overline{ \langle v_{m,k} ~ | ~ m, k \in \Z \rangle }$, the normal subgroup generated by the words $v_{m,k}.$ Similar to what we did earlier, we obtain a presentation for $( GVB_3' / W ) / V$ by inserting the relations $\alpha_{m+1,k,1} ~ = ~ \alpha_{m-1,k,1},$ $ ~ m, k \in \Z$ in the above presentation for $GVB_3' / W$; and the obtained presentation for $( GVB_3' / W ) / V$ is as follows:\\
	
	Generators: $ \{ ~ \alpha_{0,k,1}, ~ \alpha_{1,k,1} ~ | ~ k \in \Z ~ \};  $ \\
	
	Defining relations:
	\begin{equation}\label{gvbn34}
		\alpha_{0,0,1} = \alpha_{1,0,1} = 1;
	\end{equation}
	\begin{equation}\label{gvbn35}
		\alpha_{1,k,1} = \alpha_{0,k,1}^{-1}.
	\end{equation}
	
	Clearly, from \eqnref{gvbn35} we can remove the generators $\alpha_{1,k,1}$ and a free presentation for the group $(GVB_3' / W) / V$ as follows:
	\begin{equation*}
		(GVB_3' / W) / V = \langle ~ \alpha_{0,k,1}, ~ k \in \Z- \{ 0 \} ~ \rangle .
	\end{equation*}
	
	Hence, we have obtained the following quotient maps:
	\begin{equation*}GVB_3' ~ \xrightarrow{\phi} ~ GVB_3' / W ~ \xrightarrow{\psi} ~ ( GVB_3' / W ) / V ~ \cong ~ \langle ~ \alpha_{0,k,1}, ~ k \in \Z- \{ 0 \} ~ \rangle ~ \cong ~ F^{\infty}, \end{equation*}
	which gives us an onto homomorphism $~ \psi \circ \phi ~$ from the group $GVB_3'$ to the free group of infinite rank $F^{\infty}$. This proves that $GVB_3'$ is not finitely generated.
\end{proof}

\subsection{Finite generation of $GVB_4'$}

\begin{lemma}\label{gvbnlemma4}
	$GVB_4'$ is finitely generated.
\end{lemma}

\begin{proof}
	To prove this lemma, we will apply different Tietze transformations to the presentation for $GVB_4'$ deduced in \lemref{gvbnlemma2}.\\
	From \lemref{gvbnlemma2} we have the following presentation for $GVB_4'$:\\
	
	Generators: $ \{ ~ \alpha_{m,k,1}, ~ \alpha_{m,k,2}, ~ \alpha_3, ~ \beta_{m,k,2}, ~ \beta_{m,3} ~ | ~ m, k \in \Z ~ \} $.\\
	
	Defining relations: for all $m, k \in \Z,$
	\begin{equation}\label{gvbn36} \alpha_{m,k,1} ~ \alpha_{3} ~ \alpha_{m+1,k,1}^{-1} ~ \alpha_{3}^{-1} ~ = 1 ; \end{equation}
	\begin{equation}\label{gvbn37} \alpha_{m,k,1} ~ \beta_{m+1,3} ~ \alpha_{m,k+1,1}^{-1} ~ \beta_{m,3}^{-1} ~ =1 ; \end{equation}
	\begin{equation}\label{gvbn38} \alpha_{m,k,1} ~ \alpha_{m+1,k,2} ~ \alpha_{m+2,k,1} ~ \alpha_{m+2,k,2}^{-1} ~ \alpha_{m+1,k,1}^{-1} ~ \alpha_{m,k,2}^{-1} ~ =1 ; \end{equation}
	\begin{equation}\label{gvbn39} \alpha_{m,k,2} ~ \alpha_{3} ~ \alpha_{m+2,k,2} ~ \alpha_{3}^{-1} ~ \alpha_{m+1,k,2}^{-1} ~ \alpha_{3}^{-1} ~ =1 ; \end{equation}
	\begin{equation}\label{gvbn40} \beta_{m,k+1,2} ~ \beta_{m,k+2,2}^{-1} ~ \beta_{m,k,2}^{-1} ~ =1 ; \end{equation}
	\begin{equation}\label{gvbn41} \beta_{m,k,2} ~ \beta_{m,3} ~ \beta_{m,k+2,2} ~ \beta_{m,3}^{-1} ~ \beta_{m,k+1,2}^{-1} ~ \beta_{m,3}^{-1} ~  =1 ; \end{equation}
	\begin{equation}\label{gvbn42} \alpha_{m,k+1,2} ~ \alpha_{m+1,k+1,1} ~ \beta_{m+2,k,2}^{-1} ~ \alpha_{m+1,k,1}^{-1} ~ \alpha_{m,k,2}^{-1} ~ =1 ; \end{equation}
	\begin{equation}\label{gvbn43} \beta_{m,k,2} ~ \alpha_{3} ~ \alpha_{m+1,k+1,2} ~ \beta_{m+2,3}^{-1} ~ \alpha_{m+1,k,2}^{-1} ~ \alpha_{3}^{-1} ~  =1 ; \end{equation}
	\begin{equation}\label{gvbn44} \beta_{m,k,2} ~ \alpha_{m,k+1,1} ~ \alpha_{m+1,k+1,2} ~ \alpha_{m+1,k,2}^{-1} ~ \alpha_{m,k,1}^{-1} ~ =1 ; \end{equation}
	\begin{equation}\label{gvbn45} \beta_{m,3} ~ \alpha_{m,k+1,2} ~ \alpha_{3} ~ \beta_{m+2,k,2}^{-1} ~ \alpha_{3}^{-1} ~ \alpha_{m,k,2}^{-1} ~ =1 . \end{equation}
	\begin{equation}\label{gvbn48}
	\alpha_{m,0,1} = 1.
	\end{equation}
	
	\medskip From \eqnref{gvbn37}, putting $k = 1$, we get:
	\begin{equation}
		\beta_{m+1,3} = \beta_{m,3} ~ \alpha_{m,1,1}.
	\end{equation}
	Using the above relation finitely many times we get:
	\begin{equation}
		\beta_{m,3} = \beta_{0,3} ~ \alpha_{0,1,1} \dots \alpha_{m-1,1,1}, \hbox{ if } m \ge 1;
	\end{equation} 
	\begin{equation}
		\beta_{m,3} = \beta_{0,3} ~ \alpha_{-1,1,1}^{-1} ~ \dots ~ \alpha_{m,1,1}^{-1}, \hbox{ if } m \le -1.
	\end{equation}
	So, we can remove $\beta_{m,3}$, for all $m \ne 0$, from the set of generators by replacing these values in all the other relations.
	
	\medskip After the above replacement \eqref{gvbn45} becomes:
	\begin{equation}
		\beta_{0,3} ~ \alpha_{0,1,1} \dots \alpha_{m-1,1,1} ~ \alpha_{m,k+1,2} ~ \alpha_{3} ~ \beta_{m+2,k,2}^{-1} ~ \alpha_{3}^{-1} ~ \alpha_{m,k,2}^{-1} ~ =1, \hbox{ for } m \ge 1;
	\end{equation}
	\begin{equation}
		\beta_{0,3} ~ \alpha_{-1,1,1}^{-1} ~ \dots ~ \alpha_{m,1,1}^{-1} ~ \alpha_{m,k+1,2} ~ \alpha_{3} ~ \beta_{m+2,k,2}^{-1} ~ \alpha_{3}^{-1} ~ \alpha_{m,k,2}^{-1} ~ =1, \hbox{ for } m \le -1.
	\end{equation}
	\begin{equation}
	\beta_{0,3} ~ \alpha_{0,k+1,2} ~ \alpha_{3} ~ \beta_{2,k,2}^{-1} ~ \alpha_{3}^{-1} ~ \alpha_{0,k,2}^{-1} ~ =1, \hbox{ for } m = 0.
	\end{equation}
	
	From the above relations it is clear that we can express $\beta_{m,k,2}$ in terms of elements from $ \{ ~ \alpha_{m,1,1}, ~ \alpha_{m,k,2}, ~ \alpha_3, ~ \beta_{0,3} ~ | ~ m, k \in \Z ~ \} $. We remove all $\beta_{m,k,2}$ from the generating set after replacing these values in all other relations.
	
	\medskip Next we note that, after the above replacement in \eqref{gvbn40}, iterating the transformed relation finitely many times we can express each $\alpha_{m,k,2}$ in terms of elements from the set 
$$ \{ ~ \alpha_{m,0,2}, ~ \alpha_{m,1,2}, ~ \alpha_{m,2,2}, ~ \alpha_{m,1,1}, ~ \alpha_3, ~ \beta_{0,3} ~ | ~ m \in \Z ~ \}.$$ 
But using \eqref{gvbn39} we can further simplify the expression for $\alpha_{m,k,2}$, and we write $\alpha_{m,k,2}$ in terms of elements from the set 
$$ \{ ~ \alpha_{0,0,2}, ~ \alpha_{1,0,2}, ~ \alpha_{0,1,2}, ~ \alpha_{1,1,2}, ~ \alpha_{0,2,2}, ~ \alpha_{1,2,2}, ~ \alpha_3, ~ \beta_{0,3}, ~ \alpha_{m,1,1} ~ | ~ m \in \Z ~ \}.$$
 Then we replace these values of $\alpha_{m,k,2}$ in other relations and remove  $\alpha_{m,k,2}$ from the generating set.
	
	\medskip Finally we note that \eqref{gvbn36} is still unchanged after all the above replacements. Using this relation we have:
	$$\alpha_{m,k,1} = \alpha_3^{-m} ~ \alpha_{0,k,1} ~ \alpha_3^m.$$
	We replace these values in all the relations and remove all $\alpha_{m,k,1}$ with $m \ne 0$ from the generating set. Then using the relation \eqref{gvbn42} we express all $\alpha_{0,k,1}$ in terms of the elements $ \alpha_{0,0,2}, ~ \alpha_{1,0,2}, ~ \alpha_{0,1,2}, ~ \alpha_{1,1,2}, ~ \alpha_{0,2,2}, ~ \alpha_{1,2,2}, ~ \alpha_3, ~ \beta_{0,3}, ~ \alpha_{0,1,1} $ and remove all $\alpha_{0,k,1}$ except $\alpha_{0,0,1}$. But, by \eqref{gvbn48}, $\alpha_{0,0,1}=1$, hence is removed from the set of generators.
	
	\medskip Hence we have shown that $GVB_4'$ can be generated by the finite set of generators:
	$$\{ \alpha_{0,0,2}, ~ \alpha_{1,0,2}, ~ \alpha_{0,1,2}, ~ \alpha_{1,1,2}, ~ \alpha_{0,2,2}, ~ \alpha_{1,2,2}, ~ \alpha_3, ~ \beta_{0,3}, ~ \alpha_{0,1,1} \}.$$
	This completes the proof of the lemma. \end{proof}

\subsection{Finite generation of $GVB_n', ~ n \ge 5$}\label{gvb5} 

\begin{lemma}\label{gvbnlemma5}
	$GVB_n'$ has a generating set with $3n - 7$ generators, for all $n \ge 5$.
\end{lemma}

\begin{proof}
	To prove this lemma, we will apply different Tietze transformations to the presentation for $GVB_n'$, for $n \ge 5$, deduced in \lemref{gvbnlemma2}.\\
	From \eqref{gvbn16} we get:
	$$\beta_{m,k,2} ~ = \alpha_{m-1,k,1}^{-1} ~ \alpha_{m-2,k,2}^{-1} ~ \alpha_{m-2,k+1,2} ~ \alpha_{m-1,k+1,1}.$$
	We replace these values of $\beta_{m,k,2}$ in all other defining relations and remove $\beta_{m,k,2}$ from the set of generators.
	
	Note that the above substitution does not change the relation \eqref{gvbn7}, which gives:
	$$ \alpha_{m,k+1,2} = \beta_{m,j}^{-1} ~ \alpha_{m,k,2} ~ \beta_{m+1,j}, ~ j \ge 4.$$
	Here we need $n \ge 5$. Choosing $j = 4$, we replace $\alpha_{m,k,2}$ by $\beta_{m,4}^{-k} ~ \alpha_{m,0,2} ~ \beta_{m+1,4}^k$ in all other relations, and remove all $\alpha_{m,k,2}$ with $k \ne 0$ from the generating set.
	
	After this replacement \eqref{gvbn11} becomes:
	$$\beta_{m,4}^{-k} ~ \alpha_{m,0,2} ~ \beta_{m+1,4}^k ~ \alpha_{3} ~ \beta_{m+2,4}^{-k} ~ \alpha_{m+2,0,2} ~ \beta_{m+3,4}^k ~ \alpha_{3}^{-1} ~ \beta_{m+2,4}^{-k} ~ \alpha_{m+1,0,2}^{-1} ~ \beta_{m+1,4}^k ~ \alpha_{3}^{-1} ~ =1.$$
	Note that, using the above relation finitely many times, we can express $\alpha_{m,0,2}$ in terms of $\alpha_{0,0,2}, ~ \alpha_{1,0,2}, ~ \alpha_3, ~ \beta_{m,4}.$ And, we remove all $\alpha_{m,0,2}$ except $\alpha_{0,0,2}$ and $\alpha_{1,0,2}$ from the generating set.
	
	Now look at the relations \eqref{gvbn1} and \eqref{gvbn6}. Note that both the relations are untouched after all the above substitutions. The relation \eqref{gvbn6} gives us:
	$$\alpha_{m,k+1,1} = \beta_{m,j}^{-1} ~ \alpha_{m,k,1} ~ \beta_{m+1,j}$$
	Note here that, $\alpha_{m,0,1} = 1$. So for all $j \ge 3$, using the above relation finitely many times we deduce that,
	\begin{equation}\label{gvbn46}
		\alpha_{m,k,1} = \beta_{m,j}^{-k} ~ \alpha_{m,0,1} ~ \beta_{m+1,j}^k = \beta_{m,j}^{-k} ~ \beta_{m+1,j}^k 
	\end{equation}

	Now, if we put the values of $\alpha_{m,1,1}$ obtained from above relation in \eqref{gvbn1} we have:
	\begin{equation}\label{gvbn47}
	\beta_{m,j}^{-1} ~ \beta_{m+1,j} ~ \alpha_l ~ \beta_{m+2,j}^{-1} ~ \beta_{m+1,j} ~ \alpha_l^{-1} = 1, ~ ~ \hbox{for any} ~ j, ~ l \ge 3.
	\end{equation}
	
	We remove all $\alpha_{m,k,1}$ from the set of generators by replacing the values of $\alpha_{m,k,1}$ as in \eqref{gvbn46} in all the relations.
	
	And finally, for every $j \ge 3$, using \eqref{gvbn47} we can express $\beta_{m,j}$ in terms of $\beta_{0,j}, ~ \beta_{1,j}, ~ \alpha_j$. So for each $j \ge 3$, we can remove all $\beta_{m,j}$ with $m \ne 0, 1$ from the set of generators.
	
	Hence, we can generate $GVB_n'$, for all $~ n \ge 5,$ with the finite generating set:
	$$ \{ ~ \alpha_{0,0,2}, ~ \alpha_{1,0,2}, ~ \alpha_j, ~ \beta_{0,j}, ~ \beta_{1,j} ~ | ~ 3 \le j \le n-1 ~ \}$$
	which has $3n - 7$ elements.
	
	Hence, the proof of the lemma is complete.
	\end{proof}

\subsection{Perfectness of $GVB_n'$} 
\begin{lemma}\label{gvbnlemma6}
	$GVB_n'$ is perfect for $n \ge 5$. 
\end{lemma}

\begin{proof}
	We abelianize the presentation for $GVB_n'$ as in \lemref{gvbnlemma2} by inserting the relations of the type $x^{-1}y^{-1}xy = 1$ for all $x, y$ in the generating set, and obtain a presentation for $(GVB_n')^{ab}$. We will now show that $ (GVB_n')^{ab} \cong \langle 1 \rangle $.\\
	
	In the abelianized presentation we observe the following:\\
	
	(1) From \eqref{gvbn1} we get $\alpha_{m+1,k,1}=\alpha_{m,k,1}$ for all $m,k \in \Z$. This implies that $\alpha_{m,k,1}=\alpha_{0,k,1}$ for all $m,k \in \Z.$\\
	
	(2) From \eqref{gvbn2} we get $\alpha_{m+1,k,2}=\alpha_{m,k,2}$ for all $m,k \in \Z$. This implies that $\alpha_{m,k,2}=\alpha_{0,k,2}$ for all $m,k \in \Z.$ (Note that here we need $n \ge 5$.)\\
	
	(3) From \eqref{gvbn10} using the above 2 observations we get: $\alpha_{0,k,1}=\alpha_{0,k,2}$ for all $k \in \Z.$\\
	
	(4) From \eqref{gvbn11} we get: $\alpha_3 = \alpha_{0,k,2}$ for all $k \in \Z$.\\
	
	(5) From \eqref{gvbn12} we get $\alpha_i = \alpha_{i+1}$ for all $3 \le i \le n-2.$\\
	
	(6) From all the above observations we have:
	$$\alpha_{3} = \alpha_{0,0,2} = \alpha_{0,0,1} = 1$$
	$$ \implies \alpha_{m,k,1} = \alpha_{m,k,2} = \alpha_i = 1, \hbox{ for all }  m, k \in \Z, ~ 3 \le i \le n-1.$$
	
	(7) From \eqref{gvbn19} and observation (6) we get: $\beta_{m,k,2}=1$ for all $m,k \in \Z.$\\
	
	(8) From \eqref{gvbn20} and observations (6) and (7) we get: $\beta_{m,3}=1$ for all $m \in \Z.$\\
	
	(9) From \eqref{gvbn21} we get: $\beta_{m,i+1} = \beta_{m+2,i}$ for all $m \in \Z.$ Putting $i=3$ we get $\beta_{m,4} = \beta_{m+2,3} = 1$ for all $m \in \Z.$ Repeated use of this relation for increasing $i$'s gives us: $\beta_{m,i}=1$ for all $m \in \Z, ~ 3 \le i \le n-1.$\\
	
	From the above observations we conclude that, in the presentation for $(GVB_n')^{ab}$ all the generators are equal to 1. So, $(GVB_n')^{ab} \cong \langle 1 \rangle.$ Hence, $GVB_n'$ is perfect for $n \ge 5.$
\end{proof}

\begin{lemma}\label{gvb34}
$GVB_3'$ and $GVB_4'$ are not perfect. 
\end{lemma}
\begin{proof}
To prove this lemma, we shall use the concept of `adorability' of a group defined by Roushon in \cite{rou2}. We recall that a group $G$ is called \emph{adorable} if $G^i/G^{i+1}=1$ for some $i$, where $G^i=[G^{i-1}, G^{i-1}]$ and $G^0=G$ are the terms in the derived series of $G$. The smallest $i$ for which the above property holds, is called the \emph{degree of adorability} of $G$, denoted as $doa(G)$.\\

Let $k=3$ or $4$. Suppose, if possible, $GVB_k'$ is perfect. Then $GVB_k$ is adorable. It follows from \eqref{gvbndiagram}, the welded braid group $WB_k$ is a homomorphic image of $GVB_k$. It follows from \cite[Lemma 1.1]{rou2} that homomorphic image of an adorable group is adorable. But, it is proved in \cite[Corollary 1.4]{dg1}  that $WB_k$ is not adorable. Hence, $GVB_k$ is not adorable. So, in particular,  $GVB_3'$ and $GVB_4'$ are not perfect. 
\end{proof} 
\subsection{Proof of \thmref{gvbnmainth}}\label{gvb}

Combining \lemref{gvbnlemma3}, \lemref{gvbnlemma4}, \lemref{gvbnlemma5}, \lemref{gvbnlemma6}, and \lemref{gvb34}, we have the proof of \hbox{\thmref{gvbnmainth}}.\\

\section{Singular Braid Groups: Proof of \thmref{sgnmainth}}\label{sgn}

We have the following lemma.

\begin{lemma}\label{sgnlemma2}
	The group $SG_n'$ has the following presentation.
	
	Set of generators: $$ \{ ~ \alpha_{m,k,2},~ \beta_{m,k,2},~ \alpha_j,~ \beta_{m,j} ~ | ~ m, k \in \Z,~ 3 \le j \le n-1 ~ \} $$
	
	Set of defining relations:
	\medskip for all $~ m, k \in \Z$,
	\begin{equation}\label{gvbn02} \alpha_{m,k,2} ~ \alpha_{j} ~ \alpha_{m+1,k,2}^{-1} ~ \alpha_{j}^{-1} ~ =1, ~ j \ge 4 ; \end{equation}
	\begin{equation}\label{gvbn03} \alpha_{i} ~ \alpha_{j} ~ \alpha_{i}^{-1} ~ \alpha_{j}^{-1} ~   =1, ~ i,j \ge 3, ~ |i-j| > 1 ; \end{equation}
	\begin{equation}\label{gvbn04} \beta_{m,k,2} ~ \beta_{m,j} ~ \beta_{m,k+1,2}^{-1} ~ \beta_{m,j}^{-1} ~  =1, ~ j \ge 4 ; \end{equation}
	\begin{equation}\label{gvbn05} \beta_{m,i} ~ \beta_{m,j} ~ \beta_{m,i}^{-1} ~ \beta_{m,j}^{-1} ~ =1, ~ i,j \ge 3, ~ |i-j| > 1 ; \end{equation}
	\begin{equation}\label{gvbn06} \beta_{m+1,j} ~ \beta_{m,j}^{-1} ~ =1, ~ j \ge 3 ; \end{equation}
	\begin{equation}\label{gvbn07} \alpha_{m,k,2} ~ \beta_{m+1,j} ~ \alpha_{m,k+1,2}^{-1} ~ \beta_{m,j}^{-1} ~  =1, ~ j \ge 4 ; \end{equation}
	\begin{equation}\label{gvbn08} \alpha_{i} ~ \beta_{m+1,k,2} ~ \alpha_{i}^{-1} ~ \beta_{m,k,2}^{-1} ~ =1, ~ i \ge 4 ; \end{equation}
	\begin{equation}\label{gvbn09} \alpha_{i} ~ \beta_{m+1,j} ~ \alpha_{i}^{-1} ~ \beta_{m,j}^{-1} ~ =1, ~ i,j \ge 3, ~ |i-j|>1 ; \end{equation}
	\begin{equation}\label{gvbn010}  \alpha_{m+1,k,2} ~ \alpha_{m+2,k,2}^{-1} ~ \alpha_{m,k,2}^{-1} ~ =1 ; \end{equation}
	\begin{equation}\label{gvbn011} \alpha_{m,k,2} ~ \alpha_{3} ~ \alpha_{m+2,k,2} ~ \alpha_{3}^{-1} ~ \alpha_{m+1,k,2}^{-1} ~ \alpha_{3}^{-1} ~ =1 ; \end{equation}
	\begin{equation}\label{gvbn012} \alpha_{i} ~ \alpha_{i+1} ~ \alpha_{i} ~ \alpha_{i+1}^{-1} ~ \alpha_{i}^{-1} ~ \alpha_{i+1}^{-1}  =1, ~ i \ge 3; \end{equation}
	\begin{equation}\label{gvbn013} \alpha_{m,k+1,2} ~ \beta_{m+2,k,2}^{-1} ~ \alpha_{m,k,2}^{-1} ~ =1 ; \end{equation}
	\begin{equation}\label{gvbn014} \beta_{m,k,2} ~ \alpha_{3} ~ \alpha_{m+1,k+1,2} ~ \beta_{m+2,3}^{-1} ~ \alpha_{m+1,k,2}^{-1} ~ \alpha_{3}^{-1} ~  =1 ; \end{equation}
	\begin{equation}\label{gvbn015} \beta_{m,i} ~ \alpha_{i+1} ~ \alpha_{i} ~ \beta_{m+2,i+1}^{-1} ~ \alpha_{i}^{-1} ~ \alpha_{i+1}^{-1} ~ =1, ~ i \ge 3 ; \end{equation}
	\begin{equation}\label{gvbn016} \beta_{m,k,2} ~ \alpha_{m+1,k+1,2} ~ \alpha_{m+1,k,2}^{-1} ~ =1 ; \end{equation}
	\begin{equation}\label{gvbn017} \beta_{m,3} ~ \alpha_{m,k+1,2} ~ \alpha_{3} ~ \beta_{m+2,k,2}^{-1} ~ \alpha_{3}^{-1} ~ \alpha_{m,k,2}^{-1} ~ =1 ; 
	\end{equation}
	\begin{equation}\label{gvbn018} \beta_{m,i+1} ~ \alpha_{i} ~ \alpha_{i+1} ~ \beta_{m+2,i}^{-1} ~ \alpha_{i+1}^{-1} ~ \alpha_{i}^{-1} ~ =1, ~ i \ge 3; \end{equation}
	\begin{equation}\label{gvbn020} 
	\alpha_{m,k,2} ~ \beta_{m+1,k,2} ~ \alpha_{m,k+1,2}^{-1} ~ \beta_{m,k,2}^{-1} = 1; \end{equation}
	\begin{equation}\label{gvbn021} 
	\alpha_{i} ~ \beta_{m+1,i} ~ \alpha_{i}^{-1} ~ \beta_{m,i}^{-1} = 1. \end{equation}
		
\end{lemma}

\begin{proof} 
We execute Step 3 of the Reidemeister-Schreier algorithm as we did in case of $GVB_n'$ and we get the same relations:
$$S_{\sigma_1^m, \sigma_1}=1, ~ \hbox{ i.e. } ~ \alpha_{m,0,1}=1;$$
$$S_{\sigma_1^m \rho_1^k, \rho_1}=1, ~ \hbox{ i.e. } ~ \beta_{m,k,1}=1.$$

Note that $SG_n$ and $GVB_n$ both have the braid relations among $\sigma_i$, and the relations \eqnref{mm5}, \eqnref{mm2}, \eqnref{mm3}, \eqnref{mm4} common in their presentations. So, certain defining relations in $SG_n'$ are obtained immediately from the proof of \lemref{gvbnlemma2}. What remains to do is to re-write the defining relations for $SG_n$ which are disjoint from the set of defining relations for $GVB_n$. In the following we sketch all the computations for completeness. 

\medskip Following Step 4 of the Reidemeister-Schreier algorithm, we calculate the terms $\tau(\lambda r_{\mu} \lambda^{-1})$ for each $\lambda \in \Lambda$ and for each of the defining relations $r_{\mu}=1$ of $SG_n$, as follows:
$$r_1 = \sigma_i \sigma_j \sigma_i^{-1} \sigma_j^{-1} = 1, \ |i-j|>1;$$
$$r_2 = \rho_i \rho_j \rho_i^{-1} \rho_j^{-1} = 1, \ |i-j|>1;$$
$$r_3 = \sigma_i \rho_j \sigma_i^{-1} \rho_j^{-1} = 1, \ |i-j|>1;$$
$$r_4 = \sigma_i \sigma_{i+1} \sigma_i \sigma_{i+1}^{-1} \sigma_i^{-1} \sigma_{i+1}^{-1}=1;$$
$$r_6 = \rho_i \sigma_{i+1} \sigma_i \rho_{i+1}^{-1} \sigma_i^{-1} \sigma_{i+1}^{-1} = 1;$$
$$r_7 = \rho_{i+1} \sigma_i \sigma_{i+1} \rho_i^{-1} \sigma_{i+1}^{-1} \sigma_i^{-1} = 1;$$
$$r_8 = \sigma_i \rho_i \sigma_i^{-1} \rho_i^{-1} = 1.$$

\medskip Thus we get the following set of defining relations for $SG_n'$.
\begin{equation}\label{eq01}
\alpha_{m,0,1}=1;
\end{equation}
\begin{equation}\label{eq02}
\beta_{m,k,1}=1;
\end{equation}
\begin{equation}\label{eq03}
\alpha_{m,k,i} ~ \alpha_{m+1,k,j} ~ \alpha_{m+1,k,i}^{-1} ~ \alpha_{m,k,j}^{-1} = 1,~ |i-j|>1;
\end{equation}
\begin{equation}\label{eq04}
\beta_{m,k,i} ~ \beta_{m,k+1,j} ~ \beta_{m,k+1,i}^{-1} ~ \beta_{m,k,j}^{-1} = 1,~ |i-j|>1;
\end{equation}
\begin{equation}\label{eq05}
\alpha_{m,k,i} ~ \beta_{m+1,k,j} ~ \alpha_{m,k+1,i}^{-1} ~ \beta_{m,k,j}^{-1} = 1,~ |i-j|>1;
\end{equation}
\begin{equation}\label{eq06}
\alpha_{m,k,i} ~ \alpha_{m+1,k,i+1} ~ \alpha_{m+2,k,i} ~ \alpha_{m+2,k,i+1}^{-1} ~ \alpha_{m+1,k,i}^{-1} ~ \alpha_{m,k,i+1}^{-1} = 1;
\end{equation}
\begin{equation}\label{eq07}
		\beta_{m,k,i} ~ \alpha_{m,k+1,i+1} ~ \alpha_{m+1,k+1,i} ~ \beta_{m+2,k,i+1}^{-1} ~ \alpha_{m+1,k,i}^{-1} ~ \alpha_{m,k,i+1}^{-1} = 1;
	\end{equation}
	\begin{equation}\label{eq08}
		\beta_{m,k,i+1} ~ \alpha_{m,k+1,i} ~ \alpha_{m+1,k+1,i+1} ~ \beta_{m+2,k,i}^{-1} ~ \alpha_{m+1,k,i+1}^{-1} ~ \alpha_{m,k,i}^{-1} = 1;
	\end{equation}
	\begin{equation}\label{eq09}
\alpha_{m,k,i} ~ \beta_{m+1,k,i} ~ \alpha_{m,k+1,i}^{-1} ~ \beta_{m,k,i}^{-1} = 1.
\end{equation}

As we did in the proof of \lemref{gvbnlemma2}, we do similar computations to observe that:
$$\alpha_{m,k,j} = \alpha_{0,0,j},~ j \ge 3,~ m,k \in \Z;$$
$$\beta_{m,k,j} = \beta_{m,0,j},~ j \ge 3, ~ m,k \in \Z.$$

Using the above relations, we replace $\alpha_{m,k,j}$ by $\alpha_{0,0,j}$ for all $(m, k) \ne (0,0)$ and $j \ge 3$, in all the defining relations, and remove all $\alpha_{m,k,j}$ with $(m, k) \ne (0,0),~ j \ge 3$, from the set of generators. After this replacement we denote $\alpha_{0,0,j}$ simply by $\alpha_j$ for $j \ge 3$.\\

Similarly, we replace $\beta_{m,k,j}$ by $\beta_{m,0,j}$ for all $k \ne 0,~ j \ge 3,~ m \in \Z,$ in all the defining relations, and remove all $\beta_{m,k,j}$ with $k \ne 0, ~ j \ge 3, ~ m \in \Z,$ from the set of generators. After this replacement we denote $\beta_{m,0,j}$ simply by $\beta_{m,j}$ for $j \ge 3, ~ m \in \Z.$\\

Using \eqref{eq02} we replace $\beta_{m,k,1}$ by 1, for all $m,k \in \Z$, in all the defining relations, and remove all $\beta_{m,k,1}$ from the set of generators.\\

Finally, putting $i=1$ in \eqref{eq09} and using \eqref{eq01} we get: $\alpha_{m,k,1} = 1 $ for all $m,k \in \Z.$ We replace $\alpha_{m,k,1}$ by 1, for all $m,k \in \Z$, in all the defining relations, and remove all $\alpha_{m,k,1}$ from the set of generators.\\

After incorporating all the replacements and notation changes as described above we get the following set of generators for $SG_n'$:
$$\{ ~ \alpha_{m,k,2},~ \beta_{m,k,2},~ \alpha_j,~ \beta_{m,j} ~ | ~ m, k \in \Z,~ 3 \le j \le n-1 ~ \},$$

and the following list of defining relations for $SG_n'$.\\

For \eqref{eq03} we have the following 3 possible cases:\\
Case 1: $i=1, ~ j \ge 3 ~$ or, $j=1, ~ i \ge 3$; gives no nontrivial relation.\\
Case 2: $i=2, ~ j \ge 4 ~$ or, $j=2, ~ i \ge 4$; gives the relations: \eqref{gvbn02}.\\
Case 3: $i,j \ge 3, ~ |i-j|>1$; gives the relations: \eqref{gvbn03}.\\

For \eqref{eq04} we have the following 3 possible cases:\\
Case 1: $i=1, ~ j \ge 3 ~$ or, $j=1, ~ i \ge 3$; gives no nontrivial relation.\\
Case 2: $i=2, ~ j \ge 4 ~$ or, $j=2, ~ i \ge 4$; gives the relations: \eqref{gvbn04}.\\
Case 3: $i,j \ge 3, ~ |i-j|>1$; gives the relations: \eqref{gvbn05}.\\

For \eqref{eq05} we have the following 5 possible cases:\\
Case 1: $i=1, ~ j \ge 3$; gives the relations: \eqref{gvbn06}.\\
Case 2: $j=1, ~ i \ge 3$; gives no nontrivial relation.\\
Case 3: $i=2, ~ j \ge 4$; gives the relations: \eqref{gvbn07}.\\
Case 4: $j=2, ~ i \ge 4$; gives the relations: \eqref{gvbn08}.\\
Case 5: $i,j \ge 3, ~ |i-j|>1$; gives the relations: \eqref{gvbn09}.\\

For \eqref{eq06} we have the following 3 possible cases:\\
Case 1: $i=1$; gives the relations: \eqref{gvbn010}.\\
Case 2: $i=2$; gives the relations: \eqref{gvbn011}.\\
Case 3: $i \ge 3$; gives the relations: \eqref{gvbn012}.\\

	For \eqref{eq07} we have the following 3 possible cases:\\
	Case 1: $i=1$; gives the relations: \eqnref{gvbn013}.\\
	Case 2: $i=2$; gives the relations: \eqnref{gvbn014}.\\
	Case 3: $i \ge 3$; gives the relations: \eqnref{gvbn015}.\\

	For \eqref{eq08} we have the following 3 possible cases:\\
	Case 1: $i=1$; gives the relations: \eqnref{gvbn016}.\\
	Case 2: $i=2$; gives the relations: \eqnref{gvbn017}.\\
	Case 3: $i \ge 3$; gives the relations: \eqnref{gvbn018}.\\

	For \eqref{eq09} we have the following 3 possible cases:\\
	Case 1: $i=1$; gives no nontrivial relation.\\
	Case 2: $i=2$; gives the relations: \eqnref{gvbn020}.\\
	Case 3: $i \ge 3$; gives the relations: \eqnref{gvbn021}.\\

This completes the proof of the lemma.\end{proof}

We have the following lemma.

\begin{lemma}\label{sgnlemma3}
	$SG_n'$ is finitely generated for all $n \ge 5$.
\end{lemma}

\begin{proof}
	To prove this lemma, we will apply different Tietze transformations to the presentation for $SG_n'$ deduced in \lemref{sgnlemma2}.\\
	
	From \eqref{gvbn08} we get:
	$$ \beta_{m,k,2} ~ = ~ \alpha_{4}^{-m} ~ \beta_{0,k,2} ~ \alpha_{4}^m.$$
	
	\noindent From \eqref{gvbn04} we get:
	$$\beta_{m,k,2} ~ = ~ \beta_{m,4}^{-k} ~ \beta_{m,0,2} ~ \beta_{m,4}^k.$$
	
	Using the above 2 sets of relations we can express each $\beta_{m, k, 2}$ in terms of $\beta_{0,0,2}, ~ \alpha_4, ~ \beta_{m,4}$. Note that, here we need $n \ge 5$. We replace these values of $\beta_{m,k,2}$ in all other defining relations and remove all $\beta_{m,k,2}$, except $\beta_{0,0,2}$, from the set of generators.
	
	From \eqref{gvbn06} we get: for each $j \ge 3$, $\beta_{m,j}=\beta_{0,j}$ for all $m \in \Z$.
	
	We replace $\beta_{m,j}$ by $\beta_{0,j}$ in all other relations, and remove all $\beta_{m,j}$ with $m \ne 0$ from the generating set.\\
	
	After the above substitution, the relation \eqref{gvbn07} gives:
	$$ \alpha_{m,k+1,2} = \beta_{0,j}^{-1} ~ \alpha_{m,k,2} ~ \beta_{0,j}, ~ j \ge 4.$$
	Choosing $j = 4$, we replace $\alpha_{m,k,2}$ by $\beta_{0,4}^{-k} ~ \alpha_{m,0,2} ~ \beta_{0,4}^k$ in all other relations, and remove all $\alpha_{m,k,2}$ with $k \ne 0$ from the generating set.\\
	
	After this replacement \eqref{gvbn010} becomes:
	$$\beta_{0,4}^{-k} ~ \alpha_{m+1,0,2} ~ \beta_{0,4}^k ~ \beta_{0,4}^{-k} ~ \alpha_{m+2,0,2}^{-1} ~ \beta_{0,4}^k ~ \beta_{0,4}^{-k} ~ \alpha_{m,0,2}^{-1} ~ \beta_{0,4}^k ~ =1$$
	$$\iff \alpha_{m+1,0,2} ~ \alpha_{m+2,0,2}^{-1} ~ \alpha_{m,0,2}^{-1} ~ =1.$$
	Note that, using the above relation finitely many times, we can express $\alpha_{m,0,2}$ in terms of $\alpha_{0,0,2}$ and $\alpha_{1,0,2}.$ And, we remove all $\alpha_{m,0,2}$ except $\alpha_{0,0,2}$ and $\alpha_{1,0,2}$ from the generating set.

	Hence, we can generate $SG_n'$, for all $~ n \ge 5,$ with the finite generating set:
	$$ \{ ~ \alpha_{0,0,2}, ~ \alpha_{1,0,2}, ~ \alpha_j, ~ \beta_{0,j} ~ | ~ 3 \le j \le n-1 ~ \},$$
	containing $2n-4$ elements.\\
	
	Hence, proof of the lemma is complete.
\end{proof}

Now, we prove the following lemma.

\begin{lemma}\label{sgnlemma4}
	$SG_n'$ is perfect for all $n \ge 5$.
\end{lemma}

\begin{proof}
	Let $n \ge 5$. We abelianize the presentation for $SG_n'$ as given in \lemref{sgnlemma2} by inserting the commuting relations $~xy=yx~$ into the set of defining relations, for all $x,y \in \{ ~ \alpha_{m,k,2}, ~ \beta_{m,k,2}, ~ \alpha_j, ~ \beta_{m,j} ~ | ~ m,k \in \Z, ~ 3 \le j \le n-1 \}$. We observe the following.\\
	
	From \eqref{gvbn02} we get: $\alpha_{m+1,k,2} = \alpha_{m,k,2},~ \forall m,k \in \Z \implies \alpha_{m,k,2} = \alpha_{0,k,2},~ \forall m,k \in \Z.$\\
	We replace $\alpha_{m,k,2}$ by $\alpha_{0,k,2}$ in all the relations, and we remove all $\alpha_{m,k,2}$ from the set of generators except $\alpha_{0,k,2}$.\\

	From \eqref{gvbn011} we get: $\alpha_{0,k,2} = \alpha_3,~ \forall k \in \Z.$\\
	From \eqref{gvbn012} we get: $ \alpha_3 = \alpha_4 = \dots = \alpha_{n-1}.$\\
	From \eqref{gvbn010} we get: $\alpha_{0,k,2} = 1$.\\
	
	Hence we have $~ 1 = \alpha_{0,k,2} = \alpha_j$ for all $k \in \Z, ~ 3 \le j \le n-1.$ We replace all $\alpha_{0,k,2},~ \alpha_j$ by 1 in all defining relations, and remove them from the set of generators.\\
	
	From \eqref{gvbn04} and \eqref{gvbn08} we get: $\beta_{m,k,2} = \beta_{0,0,2}, ~ \forall m, k \in \Z.$\\
	From \eqref{gvbn06} we get: $\beta_{m+1,j} = \beta_{m,j}, ~ \forall m \in \Z,~ 3 \le j \le n-1.$\\ $ \hbox{ \ \ \ \ \ \ \ \ \ \ \ \ \ \ \ \ \ \ \ \ \ \ \ \ \ \ \ \ \ \ } \implies ~ \beta_{m,j} = \beta_{0,j}, ~ \forall m \in \Z,~ 3 \le j \le n-1.$\\
	Hence, we replace all $ ~ \beta_{m,k,2},~ \beta_{m,j} ~$ by $\beta_{0,0,2}, ~ \beta_{0,j}$ respectively, in all defining relations, and remove all $ ~ \beta_{m,k,2},~ \beta_{m,j} ~$ from the set of generators except $\beta_{0,0,2}$, $~ \beta_{0,j}$.\\
	
	From \eqref{gvbn014} we get: $\beta_{0,0,2}=\beta_{0,3}.$\\
	From \eqref{gvbn015} we have: $\beta_{0,3}=\beta_{0,4}= \dots = \beta_{0,n-1}.$\\
	From \eqref{gvbn013} we have: $\beta_{0,0,2}=1.$\\
	So, we have $~ 1 = \beta_{0,0,2} = \beta_{0,j}$ for all $~ 3 \le j \le n-1.$ Hence, we replace $\beta_{0,0,2}, ~ \beta_{0,j}$ by 1 in all defining relations, and remove them from the set of generators.\\
	
	After all these replacements, clearly, we are left with the trivial group. Hence, for $n \ge 5$, $(SG_n')^{ab} \cong \{1\}$, i.e. $SG_n'$ is perfect. This proves the lemma. \end{proof}

We have the following lemma.

\begin{lemma}\label{sgnlemma5}
	$SG_3'/SG_3'' $ is isomorphic to the free abelian group of infinite rank $~ \bigoplus \limits_{i=1}^{\infty} \Z ~$.
\end{lemma}

\begin{proof}
	Note that, $SG_3'$ has the following presentation:\\
	
	set of generators: $ \{ ~ \alpha_{m,k,2},~ \beta_{m,k,2} ~ | ~ m,k \in \Z ~ \}, $\\
	
	set of defining relations: for all $m,k \in \Z,$
	
	$$ \alpha_{m+1,k,2} ~ \alpha_{m+2,k,2}^{-1} ~ \alpha_{m,k,2}^{-1} ~ =1;$$
	$$\alpha_{m,k+1,2} ~ \beta_{m+2,k,2}^{-1} ~ \alpha_{m,k,2}^{-1} ~ =1;$$
	$$\beta_{m,k,2} ~ \alpha_{m+1,k+1,2} ~ \alpha_{m+1,k,2}^{-1} ~ =1;$$
	$$\alpha_{m,k,2} ~ \beta_{m+1,k,2} ~ \alpha_{m,k+1,2}^{-1} ~ \beta_{m,k,2}^{-1} = 1.$$

	We abelianize the above presentation for $SG_3'$ by inserting the commuting relations $~xy=yx~$ into the set of defining relations, for all $x,y \in \{ ~ \alpha_{m,k,2}, ~ \beta_{m,k,2} ~ | ~ m,k \in \Z \}$. In the abelianization $(SG_3')^{ab} = SG_3'/SG_3''$ we observe the following.\\
	
	We have $~ \beta_{m,k,2} ~ \alpha_{m+1,k+1,2} ~ \alpha_{m+1,k,2}^{-1} ~ =1 \iff \beta_{m,k,2} = \alpha_{m+1,k,2} ~ \alpha_{m+1,k+1,2}^{-1}~.$\\
	We replace all $\beta_{m,k,2}$ by $\alpha_{m+1,k,2} ~ \alpha_{m+1,k+1,2}^{-1}$ in all the defining relations, and remove all $\beta_{m,k,2}$ from the set of generators. So, apart from all the commuting relations, we get the following defining relations for $(SG_3')^{ab}$.
	
	$$ \alpha_{m+1,k,2} ~ \alpha_{m+2,k,2}^{-1} ~ \alpha_{m,k,2}^{-1} ~ =1;$$
	$$\alpha_{m+3,k,2} ~ \alpha_{m+3,k+1,2}^{-1} = \alpha_{m,k,2}^{-1} ~ \alpha_{m,k+1,2};$$
	$$\alpha_{m,k,2} ~ \alpha_{m+2,k,2} ~ \alpha_{m+2,k+1,2}^{-1} ~ \alpha_{m,k+1,2}^{-1} ~ \alpha_{m+1,k+1,2} ~ \alpha_{m+1,k,2}^{-1} =1.$$
	
	Note that applying the relations: $ \alpha_{m+1,k,2} ~ \alpha_{m+2,k,2}^{-1} ~ \alpha_{m,k,2}^{-1} ~ =1$, all the other relations above become trivial relations. Hence we get the following presentation for $(SG_3')^{ab}:$
	
$$\langle ~ \alpha_{m,k,2},~ m, k \in \Z ~ | ~ \alpha_{m+2,k,2}=\alpha_{m+1,k,2} ~ \alpha_{m,k,2}^{-1}~, ~ [\alpha_{m,k,2},~\alpha_{m',k',2}] = 1, ~ \forall m, m', k, k' \in \Z  ~ \rangle,$$
	where $[x,y] = x^{-1} y^{-1} x y .$
	
	Finally, we note that using the relations $\alpha_{m+2,k,2}=\alpha_{m+1,k,2} ~ \alpha_{m,k,2}^{-1}$, for each $k \in \Z,$ we can express any $\alpha_{m,k,2}$ in terms of $\alpha_{0,k,2}$ and $\alpha_{1,k,2}$. So we remove all $\alpha_{m,k,2}$ with $m \ne 0, 1,$ from the set of generators after replacing them in terms of $\alpha_{0,k,2}$ and $\alpha_{1,k,2}$ in all the surviving defining relations. And we find that $(SG_3')^{ab}$ has the following presentation:
	$$\langle ~ \alpha_{0,k,2}, ~ \alpha_{1,k,2},~ k \in \Z ~ | ~ [\alpha_{m,k,2},~\alpha_{m',k',2}] = 1, ~ \forall m, m' \in \{0,1\},~ \forall k, k' \in \Z  ~ \rangle,$$
	which is a presentation for the free abelian group of countably infinite rank, $~ \bigoplus \limits_{i=1}^{\infty} \Z ~$.
	
	Hence the proof is complete. 
	\end{proof}

We have the following lemmas.

\begin{lemma}\label{sgnlemma6}
	$SG_3'$ is not perfect.
\end{lemma}

\begin{proof}
	From \lemref{sgnlemma5}, we know that $SG_3'/SG_3''$ is not trivial. Hence $SG_3' \ncong SG_3''$. Hence $SG_3'$ is not perfect.
\end{proof}

\begin{lemma}\label{sgnlemma7}
	$SG_3'$ is not finitely generated.
\end{lemma}

\begin{proof}
	From \lemref{sgnlemma5}, it follows that $~ \bigoplus \limits_{i=1}^{\infty} \Z ~$ is a quotient group of $SG_3'$. As $~ \bigoplus \limits_{i=1}^{\infty} \Z ~$ is not finitely generated, it follows that $SG_3'$ is also not finitely generated.
\end{proof}

\begin{lemma}\label{sgnlemma8}
	$SG_4'$ is not finitely generated, and also not perfect.
\end{lemma}

\begin{proof}
	
	From \lemref{sgnlemma2} we have the following presentation for $SG_4'$:\\
	
	set of generators: $ \{ ~ \alpha_{m,k,2}, ~ \alpha_3, ~ \beta_{m,k,2}, ~ \beta_{m,3} ~ | ~ m, k \in \Z ~ \} $,\\
	
	set of defining relations: for all $m, k \in \Z,$
	$$\beta_{m+1,3} ~ \beta_{m,3}^{-1} ~ =1 ;$$
	$$\alpha_{m+1,k,2} ~ \alpha_{m+2,k,2}^{-1} ~ \alpha_{m,k,2}^{-1} ~ =1 ; $$
	$$\alpha_{m,k,2} ~ \alpha_{3} ~ \alpha_{m+2,k,2} ~ \alpha_{3}^{-1} ~ \alpha_{m+1,k,2}^{-1} ~ \alpha_{3}^{-1} ~ =1 ; $$
	$$\alpha_{m,k+1,2} ~ \beta_{m+2,k,2}^{-1} ~ \alpha_{m,k,2}^{-1} ~ =1 ; $$
	$$\beta_{m,k,2} ~ \alpha_{3} ~ \alpha_{m+1,k+1,2} ~ \beta_{m+2,3}^{-1} ~ \alpha_{m+1,k,2}^{-1} ~ \alpha_{3}^{-1} ~  =1 ; $$
	$$ \beta_{m,k,2} ~ \alpha_{m+1,k+1,2} ~ \alpha_{m+1,k,2}^{-1} ~ =1 ; $$
	$$ \beta_{m,3} ~ \alpha_{m,k+1,2} ~ \alpha_{3} ~ \beta_{m+2,k,2}^{-1} ~ \alpha_{3}^{-1} ~ \alpha_{m,k,2}^{-1} ~ = 1; $$
	$$ \alpha_{m,k,2} ~ \beta_{m+1,k,2} ~ \alpha_{m,k+1,2}^{-1} ~ \beta_{m,k,2}^{-1} = 1; $$
	$$ \alpha_{3} ~ \beta_{m+1,3} ~ \alpha_{3}^{-1} ~ \beta_{m,3}^{-1} = 1. $$

	Note that, if we replace all $\alpha_3$ and $\beta_{m,3}$ by $1$ in the above presentation, we get exactly same presentation as we have for $SG_3'$.  Hence, $SG_3'$ is isomorphic to the quotient of $SG_4'$ by the normal subgroup generated by the set $\{~ \alpha_3, ~ \beta_{m,3} ~|~ m \in \Z ~ \}$.\\
	
	Using \lemref{sgnlemma7} we conclude that $SG_4'$ is not finitely generated.\\
	
	Using \lemref{sgnlemma6} we conclude that $SG_4'$ is not perfect.
\end{proof}

\subsection*{Proof of \thmref{sgnmainth}} is obtained by combining \lemref{sgnlemma3}, \lemref{sgnlemma4}, \lemref{sgnlemma6}, \lemref{sgnlemma7}, and \lemref{sgnlemma8}.


\begin{thebibliography}{99}

\bibitem[B92]{bae}
John C. Baez. 
\newblock  Link invariants of finite type and perturbation theory
\newblock {\em Lett. Math. Phys.}, 26(1):43--51, 1992.

\bibitem[Ba04]{ba1}
Valerij G. Bardakov.
\newblock  The virtual and universal braids
\newblock {\em Fund. Math.}, Volume 184,  Pages 1--18, 2004.

\bibitem[BB09]{ba}
Valerij G. Bardakov and Paolo Bellingeri.
\newblock Combinatorial properties of virtual braids.
\newblock {\em Topology Appl.}, Volume 156, Issue 6, Pages 1071-1082, 2009.

\bibitem[BB09]{bb}
Valerij G. Bardakov and Paolo Bellingeri.
\newblock Combinatorial properties of virtual braids.
\newblock {\em Topology Appl.}, Volume 156, Issue 6, Pages 1071-1082, 2009.

\bibitem[BGN18]{bgn1}
Valeriy~G. Bardakov, Krishnendu Gongopadhyay, and Mikhail~V. Neshchadim.
\newblock Commutator subgroups of virtual and welded braid groups.
\newblock {\em Int J. Algebra Comput.}, Online first, doi: 10.1142/S0218196719500127,  2018. 
\bibitem[BB05]{bibr}
Joan~S. Birman and Tara~E. Brendle.
\newblock Braids: a survey.
\newblock In {\em Handbook of knot theory}, pages 19--103. Elsevier B. V.,
  Amsterdam, 2005.

\bibitem[Bel04]{bel1}
Paolo Bellingeri.
\newblock On presentations of surface braid groups.
\newblock {\em J. Algebra}, 274(2):543--563, 2004.


\bibitem[Bel12]{bel}
Paolo Bellingeri.
\newblock {\em Topological generalizations of braid groups: combinatoric
properties and applications to knot theory}. 
\newblock M\'emoire d'HDR soutenue \`a Caen le 10 d\'ecembre 2012.

\bibitem[BG12]{bel2}
Paolo Bellingeri and Sylvain Gervais.
\newblock Surface framed braids.
\newblock {\em Geom. Dedicata}, 159:51--69, 2012.

\bibitem[Bi93]{bir}
Joan S. Birman. 
\newblock New points of view in knot theory. 
\newblock {\em Bull. Amer. Math. Soc. (N. S.)}, 28(2):253--287, 1993.

\bibitem[Dam17]{dam}
Celeste Damiani.
\newblock A journey through loop braid groups.
\newblock {\em Expo. Math.}, 35(3):252--285, 2017.

\bibitem[DG18]{dg1}
Soumya Dey and Krishnendu Gongopadhyay.
\newblock Commutator subgroups of welded braid groups.
\newblock {\em Topology Appl.}, 237:7--20, 2018.

\bibitem[Fan15]{fang}
Xin Fang.
\newblock Generalized virtual braid groups, quasi-shuffle product and quantum
  groups.
\newblock {\em Int. Math. Res. Not. IMRN}, (6):1717--1731, 2015.

\bibitem[FKV05]{fkv} 
Roger Fenn,  Louis H. Kauffman, and  Vassily O. Manturov, 
\newblock Virtual knot theory---unsolved problems.
\newblock {\em Fund. Math. }, 188:293--323, 2005.

\bibitem[FKR98]{fkr}
Roger Fenn, Ebru Keyman, and Colin Rourke. 
\newblock The singular braid monoid embeds in a group.
\newblock {\em J. Knot Theory Ramifications}, Volume 7, no. 7, 881--892, 1998.

\bibitem[GL69]{gl}
E.~A. Gorin and V.~Ja. Lin.
\newblock Algebraic equations with continuous coefficients, and certain
  questions of the algebraic theory of braids.
\newblock {\em Mat. Sb. (N.S.)}, 78 (120):579--610, 1969.

\bibitem[Kau99]{lk1}
Louis~H. Kauffman.
\newblock Virtual knot theory.
\newblock {\em European J. Combin.}, 20(7):663--690, 1999.

\bibitem[KL04]{kala}
Louis~H. Kauffman and Sofia Lambropoulou.
\newblock Virtual braids.
\newblock {\em Fund. Math.}, 184:159--186, 2004.

\bibitem[MKS04]{mks}
Wilhelm Magnus, Abraham Karrass, and Donald Solitar.
\newblock {\em Combinatorial group theory}.
\newblock Dover Publications, Inc., Mineola, NY, second edition, 2004.
\newblock Presentations of groups in terms of generators and relations.

\bibitem[MR]{mr}
Jamie Mulholland and Dale Rolfsen.
\newblock Local indicability and commutator subgroups of {A}rtin groups.
\newblock {\em ArXiv}, arXiv:math/0606116.

\bibitem[Ore12]{orevkov}
S.~Yu. Orevkov.
\newblock On the commutants of {A}rtin groups.
\newblock {\em Dokl. Akad. Nauk}, 442(6):740--742, 2012.

\bibitem[Par09]{paris}
Luis Paris.
\newblock Braid groups and {A}rtin groups.
\newblock In {\em Handbook of {T}eichm\"uller theory. {V}ol. {II}}, volume~13
  of {\em IRMA Lect. Math. Theor. Phys.}, pages 389--451. Eur. Math. Soc.,
  Z\"urich, 2009.

\bibitem[Rou04]{rou2}
S.~K. Roushon.
\newblock Topology of 3-manifolds and a class of groups. {II}.
\newblock {\em Bol. Soc. Mat. Mexicana (3)}, 10(Special Issue):467--485, 2004.

\bibitem[Sav93]{sa}
A.~G. Savushkina.
\newblock On the commutator subgroup of the braid group.
\newblock {\em Vestnik Moskov. Univ. Ser. I Mat. Mekh.}, (6):11--14, 118
  (1994),  translation in
{\em Moscow Univ. Math. Bull.},  48(6): 9–11,  1993.

\bibitem[Ver98]{vve0}
V.~V. Vershinin.
\newblock On homological properties of singular braids.
\newblock {\em  Trans. Amer. Math. Soc. }, 350(6): 2431--2455, 1998. 

\bibitem[Ver06]{vve2}
V.~V. Vershinin.
\newblock Braids, their properties and generalizations.
\newblock In {\em Handbook of algebra. {V}ol. 4}, volume~4 of {\em Handb.
  Algebr.}, pages 427--465. Elsevier/North-Holland, Amsterdam, 2006.

\bibitem[Ver14]{vve1}
V.~V. Vershinin.
\newblock About presentations of braid groups and their generalizations.
\newblock In {\em Knots in {P}oland. {III}. {P}art 1}, volume 100 of {\em
  Banach Center Publ.}, pages 235--271. Polish Acad. Sci. Inst. Math., Warsaw,
  2014.

\bibitem[Zar18]{mz}
Matthew C.~B. Zaremsky.
\newblock Symmetric automorphisms of free groups, {BNSR}-invariants, and
  finiteness properties.
\newblock {\em Michigan Math. J.}, 67(1):133--158, 2018.

\bibitem[Zin75]{zinde}
V.~M. Zinde.
\newblock Commutants of {A}rtin groups.
\newblock {\em Uspehi Mat. Nauk}, 30(5(185)):207--208, 1975.

\end{thebibliography}

\end{document}